\documentclass[letterpaper, final, titlepage, 9pt]{amsart}
\usepackage[T1]{fontenc}
\usepackage{graphicx, baskervald}
\usepackage{amssymb}

\newtheorem{thm}{Theorem}
\newtheorem{lem}{Lemma}
\newtheorem{cor}{Corollary}
\newtheorem{defn}{Definition}

\newcommand\field[1]{\mathbb{#1}}
\newcommand\R{\field{R}}
\newcommand\ZZ{\field{Z}}
\newcommand\norm[1]{\left\Vert{} #1 \right\Vert}
\newcommand\lnorm[1]{\left\vert{} #1 \right\vert}
\newcommand\seminorm[1]{\left[ #1 \right]}
\newcommand\intlo[1]{\left(#1\right]}
\newcommand\bk[1]{\left\{ #1 \right\}}
\newcommand\ip[2]{\left\langle{} #1, #2 \right\rangle}
\newcommand\D[2]{\frac{\partial{} #1}{\partial{} #2}}
\newcommand\dD[2]{\frac{d #1}{d #2}}

\newcommand{\Ind}{\mathbf{1}}
\def\sbar{\overline{s}}
\def\ubar{\overline{u}}

\def\utilde{\widetilde{u}}

\def\shat{\widehat{s}}
\def\Omegahat{\widehat{\Omega}}
\newcommand\tildebar[1]{\widetilde{\overline{#1}}}

\def\ubartilde{\tildebar{u}}
\def\J{\mathcal{J}}
\def\I{\mathcal{I}}
\def\be{\begin{equation}}
\def\ee{\end{equation}}
\let\l\ell%
\def\solnspace{B_{2, x, t}^{5/2 + 2\alpha, 5/4 + \alpha}}
\def\densityspace{B_{2, x, t}^{1/2 + 2\alpha, 1/4 + \alpha}}
\def\coeffspace{C_{x, t}^{1/2+\alpha^*, 1/4 + \alpha^*}}
\def\ivspace{B_2^{3/2 + 2\alpha}}

\def\uxxtrspace{B_2^{1/2 + \alpha}}

\def\scontrolspace{W_2^2}
\def\gcontrolspace{B_2^{1/2 + \alpha}}
\def\fcontrolspace{B_{2, x, t}^{1, 1/4+\alpha}}

\def\muspace{B_2^{1/4}}
\def\adjointsolnspace{W_2^{2,1}}
\def\chigammaspace{B_{2,x,t}^{3/2+2\alpha^*,3/4+\alpha^*}}

\begin{document}

\title[Optimal Control of Parabolic Free Boundary Problems]{Frechet Differentiability in Besov Spaces in the Optimal Control of Parabolic Free Boundary Problems}
\author{Ugur G. Abdulla \& Jonathan Goldfarb}
\address{Department of Mathematical Sciences\\Florida Institute of Technology\\Melbourne, FL 32901, abdulla@fit.edu}
\date{}

\begin{abstract}
  We consider the inverse Stefan type free boundary problem, where information
  on the boundary heat flux and density of the sources are missing and must be
  found along with the temperature and the free boundary.
We pursue
  optimal control framework where boundary heat flux, density of sources, and free
  boundary are components of the control vector.
  The optimality criteria consists of the minimization of the $L_2$-norm
  declinations of the temperature measurements at the final moment, phase
  transition temperature, and final position of the free boundary.
  We prove the Frechet differentiability in Besov spaces, and derive the formula for the Frechet differential under minimal regularity assumptions on the data.
The result implies a necessary condition for optimal control and opens the way to the application of projective gradient methods in Besov spaces for the numerical solution of the inverse Stefan problem.
\end{abstract}

\maketitle

\section{Introduction and Statement of Main Results}

\subsection{Inverse Stefan Problem}
Consider the general one-phase Stefan problem:
\begin{gather}
  (a u_x)_x + b u_x + c u - u_{t} = f,~\text{in}~\Omega\label{eq:pde-1}\\
  u(x,0) = \phi (x),~0 \leq x \leq s(0)=s_0\label{eq:pde-init}\\
  a(0,t) u_x (0,t) = g(t),~0 \leq t \leq T\label{eq:pde-bound}\\
  a(s(t),t) u_x (s(t),t) + \gamma (s(t),t) s'(t) = \chi (s(t),t),~0 \leq t \leq T\label{eq:pde-stefan}\\
  u(s(t),t) = \mu (t),~0 \leq t \leq T,\label{eq:pde-freebound}
  \intertext{where}
  \Omega = \{(x,t): 0 < x < s(t),~0 < t \leq T\}\label{eq:pde-domain-defn}
\end{gather}
where $a, b, c, f, \phi, g, \gamma, \chi, \mu$ are given functions. Assume that $f(x,t)$ and $g(t)$ are not known, where $f(x,t)$ is the density of heat sources and $g(t)$ is the heat flux at $x=0$, the left boundary of our domain. In order to find $f(x,t)$ and $g(t)$ along with $u(x,t)$ and $s(t)$, we must have additional information. Assume that we are able to measure the temperature on our domain and the position of the free boundary at the final moment $T$.
\begin{equation}
  u(x,T) = w(x),~ 0 \leq x \leq s(T) = s_*.\label{eq:pde-finaltemp}
\end{equation}
Under these conditions, we are required to solve an \emph{inverse} Stefan problem (ISP): find a tuple
\[
  \bk{u(x,t),s(t),g(t),f(x,t)}
\]
that satisfies conditions~\eqref{eq:pde-1}--\eqref{eq:pde-finaltemp}.

The inverse Stefan type free boundary problems arise in the modeling and control of phase transition processes in various fields such as thermophysics, continuum mechanics and biomedical engineering. In particular, the ISP~\eqref{eq:pde-1}--\eqref{eq:pde-finaltemp} is motivated by the biomedical engineering problem of laser ablation of biological tissues. In this context, the free boundary $s(t)$ is the ablation depth at the moment $t$. This paper is a continuation of the research program on the inverse Stefan problem which started in recent papers~\cite{abdulla13,abdulla15}.

The inverse Stefan problem first appeared in~\cite{cannon67}; the problem
discussed was the determination of a heat flux on the fixed boundary for which
the solution of the Stefan problem has a desired free boundary.
The variational approach for solving this ill-posed inverse Stefan problem was developed in~\cite{budak72,budak73,budak74}.
In~\cite{vasilev69}, the problem of finding the optimal value for the external
temperature in order to achieve a given measurement of temperature at the final
moment was considered, and existence was proven.
In~\cite{yurii80}, the Frechet differentiability and
convergence of difference schemes was proven for the same problem, and Tikhonov
regularization was suggested.

Later development of the inverse Stefan problem proceeded along two lines:
inverse Stefan problems with given phase boundaries
in~\cite{alifanov95,bell81,budak74,cannon64,carasso82,ewing79,ewing79a,goldman97,hoffman81,sherman71},
or inverse problems with unknown phase boundaries in~\cite{baumeister80,fasano77,goldman97,hoffman82,hoffman86,jochum80,jochum80a,knabner83,ladyzhenskaya68,lurye75,niezgodka79,nochetto87,primicerio82,sagues82,talenti82}.
We refer to the monograph~\cite{goldman97} for a complete list of references for both types of inverse Stefan problem, both for linear and quasilinear parabolic equations.

The main methods developed for ISP are based on a variational formulation,
Frechet differentiability, and iterative gradient methods.
The established variational methods in earlier works fail in general to address two issues~\cite{abdulla13,abdulla15}:
\begin{itemize}
  \item The solution of ISP does not depend continuously on the phase transition temperature $\mu(t)$ from~\eqref{eq:pde-freebound}. A small perturbation of the phase transition temperature may imply significant change of the solution to the inverse Stefan problem, and some sort of regularization is required.
  \item In the existing formulation, at each step of the iterative method a
    Stefan problem must be solved (that is, the unknown heat flux
    $g$ and density of sources $f$ are given, and the corresponding $u(x,t)$ and $s(t)$ are calculated) which incurs a high computational cost.
\end{itemize}

A new method developed in~\cite{abdulla13,abdulla15} addresses both issues with a new variational formulation. The key insight is that the free boundary problem has a similar nature to an inverse problem, so putting them into the same framework gives a conceptually clear formulation of the problem; proving existence, approximation, and Frechet differentiability is a resulting challenge. Existence of the optimal control and the convergence of the sequence of discrete optimal control problems to the continuous optimal control problem was proved in~\cite{abdulla13,abdulla15}. Our goal in this work is to prove the Frechet differentiability and to derive the formula for the Frechet differential under the minimal regularity assumptions on the data. This presents a significant technical challenge. Since in the new variational formulation the free boundary is treated as a control parameter, variation of the cost functional reflects the high sensitivity of the solution to the PDE problem with respect to the variation of the domain. To overcome this technical challenge we implement and widely exploit the framework of Besov spaces. In fact, this paper reveals that Besov spaces provides the natural setting for the optimal control of parabolic free boundary problem, for it expresses the optimal relation between the regularity of the data and the regularity of the control parameters and state vectors.

The Frechet differentiability result implies a necessary condition for
optimality of a given control, and opens the way to the application of
projective gradient methods in Besov spaces for the numerical solution
of the inverse Stefan problem.
It should be noted that in~\cite{abdulla13,abdulla15} ISP is considered when additional measurement of the temperature is taken on the known boundary $x=0$, i.e.~\eqref{eq:pde-finaltemp} is replaced with the condition
\begin{equation}
  u(0,t) = \nu(t),~ 0 \leq t \leq T.\label{eq:pde-finaltemp1}
\end{equation}
The methods of this paper can be applied to ISP~\eqref{eq:pde-1}--\eqref{eq:pde-domain-defn},~\eqref{eq:pde-finaltemp1}, and only minor modification would be required to prove the Frechet differentiability and necessary condition for the optimality. In this paper we use~\eqref{eq:pde-finaltemp} as a required measurement to solve ISP.\@ In particular, this is motivated by the bioengineering problem of laser ablation of biological tissues.

The structure of the paper is as follows: in Section~\ref{notation} we define all the functional spaces. Section~\ref{ocp} formulates optimal control problem. Section~\ref{sec:main-results} describes the main results: existence of the optimal control is formulated in Theorem~\ref{thm:existence-opt-control}; Theorem~\ref{thm:gradient-result2} states the Frechet differentiability result and presents the formula for the Frechet differential;
in Corollary~\ref{optimalitycondition} we present the necessary condition for the optimal control in the form of the variational inequality.
In Section~\ref{sec:gradient-heuristic} a heuristic derivation of the Frechet differential is pursued. Section~\ref{PreliminaryResults} describes important preliminary results.
In Section~\ref{prel1} we recall the existence, uniqueness and energy estimates in Besov spaces for the Neumann problem to the second order linear parabolic PDEs.
In Section~\ref{traceembedding} we formulate optimal trace embedding results for the Besov spaces.
In Section~\ref{sec:conseq-of-energy-est} we prove three important technical lemmas.
Lemmas~\ref{lem:j-well-defined} and~\ref{lem:adjoint-solution-exists} are on the estimation of the Neumann problem, and adjoint PDE problem in respective Besov space norm.
In Lemma~\ref{lem:deltau-main-est} we prove an estimate on the increment of the state vector with respect to the control vector in a Besov space norm.
By applying Lemmas~\ref{lem:j-well-defined}-\ref{lem:deltau-main-est} in Section~\ref{sec:gradient-rigorous} we complete the proof of the main results.
Finally, conclusions are presented in Section~\ref{conclusions}.

\subsection{Notation}\label{notation}
We will use the notation
\[
  \Ind_{I}(x)
  = \begin{cases}
  1,&~ x \in I
  \\
  0,&~ x \not\in I
  \end{cases}
\]
for the indicator function of the set $I$, and $[r]$ for the integer part of the real number $r$.

We will require the notions of Sobolev-Slobodeckij or Besov spaces~\cite{besov79,besov79a,kufner77,nikolskii75}.
In this section, assume $U$ is a domain in $\R$ and denote by
\[
  Q_T = (0,1)\times \intlo{0, T}.
\]
\begin{itemize}
  \item For $\l \in \ZZ_+$, $W_p^{\l}(U)$ is the Banach space of measurable functions with finite norm
  \begin{align}
    \norm{u}_{W_p^{\l}(U)}
      & := \left(\int_U \sum_{k=0}^{\l} \lnorm{\dD{^k u}{x^k}}^p\,dx\right)^{1/p}\nonumber
  \end{align}
  \item For $\l \not\in \ZZ_+$, $B_p^{\l}(U)$ is the Banach space of measurable functions with finite norm

  \begin{align}
    \norm{u}_{B_p^{\l}(U)}
      & :=
    \norm{u}_{W_p^{[\l]}(U)}
    + \seminorm{u}_{B_p^{\l}(U)},
    ~\text{where}~\nonumber
    \\
    \seminorm{u}_{B_p^\l(U)}^p
      & :=\int_{U} \int_{U} \frac{\lnorm{
    \D{^{[\l] u(x)}}{x^{[\l]}}
    - \D{^{[\l] u(y)}}{x^{[\l]}}}^p
    }{\lnorm{x-y}^{1 + p(\l - [\l])}} \,dx \,dy.\nonumber
  \end{align}
  \item If $\l \in \ZZ_+$, the seminorm $\seminorm{u}_{B_p^{\l}(U)}$ is given by
  \begin{align}
    \seminorm{u}_{B_p^\l(U)}^p
      & :=
    \int_{-\infty}^\infty \int_{-\infty}^\infty \frac{\lnorm{
    \D{^{\l-1} u(x)}{x^{\l-1}}
    - 2 \D{^{\l-1} u\left(\frac{x+y}{2}\right)}{x^{\l-1}}
    + \D{^{\l-1} u(y)}{x^{\l-1}}
    }^p}{\lnorm{x-y}^{1+p}} \,dy \,dx\nonumber
  \end{align}\cite[thm.\ 5, p. 72]{solonnikov64}. By~\cite[\S 18, thm.\ 9]{besov79a}, it follows that $p=2$ and $\l \in \ZZ_+$, the $B_p^{\l}(U)$ norm is equivalent to the $W_p^{\l}(U)$ norm (i.e.\ the two spaces coincide.)

  \item $W_2^{1,1}({Q_T})$ is the Hilbert space of $L_2({Q_T})$ functions with a weak $x$- and $t$-derivatives belonging to $L_2({Q_T})$. The inner product in $W_2^{1,1}({Q_T})$ is

  \[
    \ip{u}{v}
    = \int_{Q_T} [u v+u_x v_x+u_t v_t] \, dx \,dt.
  \]
  \item $W_{2}^{2\l, \l}(\Omega)$, $\l=1,2,\ldots$ is the Hilbert space of $u \in L_{2}(\Omega)$ for which
  \[\D{^{r+s} u}{t^r x^s} \in L_2(\Omega),~0 \leq 2r+s \leq 2\l{}.\]
  The inner product in $W_{2}^{2\l,\l}(\Omega)$ is
  \[
    \ip{u}{v}
    =\int_{\Omega} \sum_{j=0}^{2\l} \sum_{2r+s=j}\D{^{r+s} u}{t^r x^s} \D{^{r+s} v}{t^r x^s} \, dx\,dt,
  \]
  where the summation indices $r,s \geq 0$ satisfy $2r+s=j$, $j=0, \ldots, 2\l$.
  \item Let $1 \leq p < \infty$, $\l_1,\l_2>0$. The Besov space $B_{p,x,t}^{\l_1, \l_2}(Q_T)$ is defined as the closure of the set of smooth functions under the norm
  \begin{gather*}
    \norm{u}_{B_{p,x,t}^{\l_1, \l_2}(Q_T)}
    = \left(\int_0^T \norm{u(x,t)}_{B_p^{\l_1}(0,1)}^p \,dt\right)^{1/p}
    \\
    + \left(\int_0^1 \norm{u(x,t)}_{B_p^{\l_2}[0,T]}^p \,dx\right)^{1/p}.
  \end{gather*}
  When $p=2$, if either $\l_1$ or $\l_2$ is an integer, the Besov seminorm may be replaced with the corresponding Sobolev seminorm due to equivalence of the norms.
  \item The H\"older space $C_{x,t}^{\alpha,\alpha/2}(Q_T)$ is the set of continuous functions with $[\alpha]$ $x$-derivatives and $[\alpha/2]$ $t$-derivatives, and for which the highest order $x$- and $t$-derivatives satisfy H\"older conditions of order $\alpha-[\alpha]$ and $\alpha/2-[\alpha/2]$, respectively.
\end{itemize}
\subsection{Optimal Control Problem}\label{ocp}
Fix any $\alpha>0$.
Consider the minimization of the functional
\begin{gather}
  \J(v) = \beta_0 \int_0^{s(T)} \lnorm{u(x, T; v) - w(x)}^2\,dx + \beta_1 \int_0^T \lnorm{u(s(t), t;v) - \mu(t)}^2\,dt +\nonumber
  \\
  + \beta_2 \lnorm{s(T) - s_*}^2\label{eq:functional}
\end{gather} on the control set
\begin{align}
  V_R & =\Big\{ v=(f, g, s) \in H := \fcontrolspace(D)\times \gcontrolspace[0, T]\times \scontrolspace[0,T];
  \nonumber
  \\
      & \qquad s(0) = s_0,~s'(0) = 0,
  ~g(0) = a(0, 0) \phi'(0),
  ~0 < \delta \leq s(t) \leq \l;\nonumber
  \\
      & \qquad \norm{v}_H := \max\left( \norm{f}_{\fcontrolspace(D)}, \norm{g}_{\gcontrolspace[0,T]}, \norm{s}_{\scontrolspace[0,T]}\right)\leq R
  \Big\},\label{eq:control-set}
\end{align}
where $\beta_0, \beta_1, \beta_2 \geq 0$ and $\l, \delta, R > 0$ are given. Define
\[
  D = \bk{(x,t) : 0\leq x\leq \l,~ 0\leq t\leq T}.
\]
For a given control vector $v \in V_R$, the state vector $u(x, t; v)$ is a
solution to the Neumann problem~\eqref{eq:pde-1}--\eqref{eq:pde-stefan}.
The formulated optimal control problem~\eqref{eq:functional}--\eqref{eq:control-set} will be called Problem $\I$.
Since the data appearing in the Neumann problem~\eqref{eq:pde-1}--\eqref{eq:pde-stefan} are in general non-smooth, the solutions may not exist in the classical sense.
The notion of solution is understood in a weak sense, i.e.\ for a fixed control vector $v \in V_R$, $u \in W_2^{2,1}(\Omega)$ is called a solution of the Neumann problem~\eqref{eq:pde-1}--\eqref{eq:pde-stefan} if it satisfies the equation~\eqref{eq:pde-1} and conditions~\eqref{eq:pde-init}--\eqref{eq:pde-stefan} pointwise almost everywhere.

\subsection{Statement of Main Results}\label{sec:main-results}
Let $\alpha>0$ be fixed as in~\eqref{eq:control-set}. The main results are established under the assumptions
\begin{gather}
  0 < a_0 \leq a(x,t)~\text{in}~D\label{eq:a-ellipticity-cond}
  \\
  a,a_x,b,c \in \coeffspace(D)\label{eq:datacond-coeff}
  \\
  w \in \ivspace(0,\l),\quad \phi \in \ivspace(0, s_{0}),\label{eq:datacond-iv}
  \\
  \chi,\gamma \in \chigammaspace(D)\label{eq:datacond-4}
  \\
  \mu \in \muspace[0,T],\label{eq:datacond-mu}
  \intertext{where $\alpha^* > \alpha$ is arbitrary, and $\chi$, $\phi$ satisfy the compatibility condition}
  \chi(s_{0},0) = \phi'(s_{0})a(s_{0},0).\label{eq:datacond-compat}
\end{gather}
Given a control vector $v \in V_R$, under the
conditions~\eqref{eq:a-ellipticity-cond}--\eqref{eq:datacond-compat} there
exists a unique pointwise a.e.\ solution $u \in W_2^{2,1}(\Omega)$ of the
Neumann problem~\eqref{eq:pde-1}--\eqref{eq:pde-stefan}
(\cite{ladyzhenskaya68,solonnikov64}.)
\begin{defn}\label{defn:adjoint}
  For given $v$ and $u = u(x, t; v)$, $\psi \in \adjointsolnspace(\Omega)$ is a solution to the adjoint problem if
  \begin{gather}
    \big(a\psi_x\big)_x - (b\psi)_x + c\psi + \psi_t = 0,\quad\text{in}~\Omega\label{eq:adj-pde}
    \\
    \psi(x, T) = 2\beta_0(u(x, T) - w(x)),~0 \leq x \leq s(T)\label{eq:adj-finalmoment}
    \\
    a(0, t)\psi_x(0, t) - b(0, t)\psi(0, t)=0,~0 \leq t \leq T\label{eq:adj-robin-fixed}
    \\
    \Big[a\psi_x - (b + s'(t))\psi\Big]_{x=s(t)} = 2\beta_1(u(s(t), t) - \mu(t)), ~0 \leq t \leq T\label{eq:adj-robin-free}
  \end{gather}
\end{defn}
Given a control vector $v \in V_R$ and the corresponding state vector $u \in
W_2^{2,1}(\Omega)$ there exists a unique pointwise a.e.\ solution $\phi \in
W_2^{2,1}(\Omega)$ of the adjoint problem~\eqref{eq:adj-pde}--\eqref{eq:adj-robin-free}
(\cite{ladyzhenskaya68,solonnikov64})
The main results of this work are as follows:
\begin{thm}[Existence of an Optimal Control]\label{thm:existence-opt-control}
  Problem $\I$ has a solution.
  That is,
  \[
    V_* = \bk{v \in V_R: \J(v) = J_* =: \inf_{v \in V_R} \J(v)} \neq \emptyset.
  \]
\end{thm}
\begin{thm}[Frechet Differentiability]\label{thm:gradient-result2}
  The functional $\J(v)$ is differentiable in the sense of Frechet, and the first variation is
  \begin{align}
    d \J(v)
      & = -\int_{\Omega} \psi \delta f \,dx\,dt - \int_{0}^T \psi(0,t) {\delta g}(t)\,dt +\nonumber
    \\
      & \quad + \int_0^T \left[2 \beta_1 (u-\mu) u_x
    + \psi \left(\chi_x - \gamma_x s' -\big(a u_x\big)_x\right)\right]_{x=s(t)}{\delta s}(t)\, dt - \nonumber
    \\
      & \quad -\int_0^T \big[\gamma \psi\big]_{x=s(t)}{\delta s}'(t)\,dt + \nonumber
    \\
      & \quad + \left(\beta_0\lnorm{u(s(T),T) - w(s(T))}^2 + 2\beta_2(s(T)-s_*)\right){\delta s}(T),\label{eq:gradient-full}
  \end{align}
  where $\psi$ is a solution to the adjoint problem in the sense of definition~\ref{defn:adjoint}, and $\delta v = ({\delta s}, {\delta g}, {\delta f})$ is a variation of the control vector $v \in V_R$ such that $v + \delta v \in V_R$.
\end{thm}
\begin{cor}[Optimality Condition]\label{optimalitycondition}
  If $\bar{v} = (\bar{f}, \bar{g}, \bar{s})$ is an optimal control, then the following
  variational inequality is satisfied:
  \begin{align}
    &-\int_{\Omega} \psi \left( f(t) - \bar{f}(t) \right) \,dx\,dt - \int_{0}^T \psi(0,t) \left( g(t) - \bar{g}(t) \right)\,dt +\nonumber
    \\
      & \quad + \int_0^T \left[2 \beta_1 (u-\mu) u_x
    + \psi \left(\chi_x - \gamma_x s' -\big(a u_x\big)_x\right)\right]_{x=s(t)}\left( s(t) - \bar{s}(t) \right)\, dt - \nonumber
    \\
      & \quad -\int_0^T \big[\gamma \psi\big]_{x=s(t)}\left( s'(t)-\bar{s}(t) \right){\delta s}'(t)\,dt + \nonumber
    \\
      & \quad + \left(\beta_0\lnorm{u(s(T),T) - w(s(T))}^2 + 2\beta_2(s(T)-s_*)\right)\left( s(T)-\bar{s}(T) \right)
    \geq 0\label{eq:optimality-condition}
  \end{align}
  for arbitrary $v = (f, g, s) \in V_R$.
\end{cor}

\section{Heuristic Derivation of the Frechet Differential}\label{sec:gradient-heuristic}
\def\L{\mathcal{L}}
To give a first indication of the form of the gradient, we apply the heuristic method of Lagrange-type multipliers; the rigorous proof follows in Section~\ref{sec:gradient-rigorous}. Consider the functional:
\[
  \L(g,f,s,u,\psi) = \J(v) + \int_0^T \int_0^{s(t)} \psi\left[(a u_x)_x + b u_x + c u - u_t -f\right]\,dx \,dt.
\]
Define $\delta v = (\delta s, \delta g, \delta f)$, $\bar{v} = v+\delta v = (\sbar, \bar{g}, \bar{f})$. Let $\ubar(x,t) = \ubar(x,t,\bar{v})$.
We will also denote by $\tilde{s}(t)=s(t)+\theta(t)\delta s(t)$ where $0 \leq \theta(t) \leq 1$ standing for all functions arising from application of mean value theorem in the region between $s(t)$ and $\sbar(t)$. Define
\begin{gather*}
  \shat = \min(s, \sbar),\quad 0 \leq t \leq T,\quad
  \Omegahat = \bk{(x,t):0 < x < \shat(t),~ 0 < t \leq T}
  \\
  \delta u(x,t) = \ubar(x,t) - u(x,t)~\text{in}~\Omegahat.
\end{gather*}
The increment $\delta v$ must be made in such a way that $v + \delta v \in V_R$, so in particular $\delta s(0)=0$, $\delta s'(0)=0$.
Similarly, the incremented solution $\ubar$ must satisfy the initial condition~\eqref{eq:pde-init} and corresponding boundary conditions~\eqref{eq:pde-bound} and~\eqref{eq:pde-stefan}, so in particular
\begin{equation}
  \delta u (x,0) = 0,
  \quad a(0,t)\delta u_x(0,t)=\delta g(t).\label{eq:gradient-heuristic-deltau-bc}
\end{equation}
In what follows, all terms of higher than linear order with respect to $\delta v$ will be absorbed into the expression $o(\delta v)$.
Partition the time domain as $[0,T]=T_{1}\cup T_{2}$ where
\begin{align*}
  T_1 & =\bk{t \in [0,T]: \delta s(t) <0},\quad
  T_2 = [0,T] \setminus T_1 = \bk{t \in [0,T]: \delta s(t) \geq 0}.
\end{align*}
Calculate
\begin{align}
  \Delta \L
      & =\Delta \J + \Delta I,\label{eq:lagrange-func-increment}
  \intertext{where}
  \Delta I
      & = \int_0^T \int_0^{\sbar(t)} \psi\left[(a \ubar_x)_x + b \ubar_x + c \ubar - \ubar_t -f-\delta f\right]\,dx \,dt - \nonumber
  \\
      & \quad -\int_0^T \int_0^{s(t)} \psi\left[(a u_x)_x + b u_x + c u - u_t -f\right]\,dx \,dt\label{eq:lagrange-inc-integral}
  \\
  \Delta \J
      & = \J(v+\delta v) - \J(v) = J_1 + J_2 + J_3
  \intertext{and the terms $J_1$-$J_3$ are given by}
  J_1 & =\beta_0\left[\int_0^{\sbar(T)}\lnorm{\ubar(x,T)-w(x)}^2 \,dx - \int_0^{s(T)}\lnorm{u(x,T)-w(x)}^2 \,dx\right]\nonumber
  \\
  J_2 & = \beta_1\left[\int_0^T\lnorm{\ubar(\sbar(t),t)-\mu(t)}^2 \,dt -\int_0^T\lnorm{u(s(t),t)-\mu(t)}^2 \,dt\right]
  = J_{21} + J_{22} \nonumber
  \\
  J_3 & = \beta_2\lnorm{\sbar(T)-s_*}^2 - \lnorm{s(T)-s_*}^2,\label{eq:functional-increment-1}
  \intertext{where}
  J_{21}
      & = \beta_1\int_{T_1}\lnorm{\ubar(\sbar(t),t)-\mu(t)}^2-\lnorm{u(s(t),t)-\mu(t)}^2 \,dt,\nonumber
  \\
  J_{22}
      & = \beta_1\int_{T_2}\lnorm{\ubar(\sbar(t),t)-\mu(t)}^2-\lnorm{u(s(t),t)-\mu(t)}^2 \,dt.\nonumber
\end{align}
In $J_{1}$, calculate
\begin{align}
  J_{1} & =\beta_0\int_0^{\shat(T)}\left[\lnorm{\ubar(x,T)-w(x)}^2 - \lnorm{u(x,T)-w(x)}^2\right] \,dx+\nonumber
  \\
        & \quad + \beta_0\int_{\shat(T)}^{\sbar(T)}\lnorm{\ubar(x,T)-w(x)}^2 \, dx- \beta_0\int_{\shat(T)}^{s(T)}\lnorm{u(x,T)-w(x)}^2 \, dx.\nonumber
\end{align}
The first term comprises the increment of a quadratic functional, while
\begin{gather*}
  \int_{\shat(T)}^{\sbar(T)} \lnorm{\ubar(x,T)-w(x)}^2\,dx
  = \Ind_{T_2}(T)\lnorm{u(s(T),T) - w(s(T))}^2\delta s(T) + o(\delta v)\nonumber
  \\
  \int_{\shat(T)}^{s(T)} \lnorm{u(x,T)-w(x)}^2\,dx
  = -\Ind_{T_1}(T)\lnorm{u(s(T),T) - w(s(T))}^2\delta s(T) 
  + o(\delta v).\nonumber
\end{gather*}
Hence
\begin{gather}
  J_{1} = \beta_0\int_0^{\shat(T)} \big[2(u-w) \delta u\big]_{t=T}\,dx + \beta_0\lnorm{u(s(T),T) - w(s(T))}^2 \delta s(T)\nonumber
  \\
  \quad + o(\delta v).\label{eq:gradient-heuristic-split-J-1-final}
\end{gather}
Using the identity for $t \in T_1$
\begin{align}
  \ubar(\sbar(t),t)
    & =u(s(t),t) + u_x(\tilde{s}(t),t)\delta s(t) + \delta u(\sbar(t),t)\nonumber
  \intertext{in $J_{21}$, it follows that}
  J_{21}
    & =\beta_1\int_{T_1}2\big( u(s(t),t)-\mu(t)\big)u_x(s(t),t)\delta s(t) \,dt +\nonumber
  \\
    & \quad + \beta_1\int_{T_1}2\big( u(s(t),t)-\mu(t)\big)\delta u(\sbar(t),t) \,dt + o(\delta v).\label{eq:gradient-heuristic-split-J-21}
\end{align}
Similarly, use the identity for $t \in T_2$
\begin{align}
  \ubar(\sbar(t),t)
    & =u(s(t),t) + \delta u(s(t),t) + \ubar_x(\tilde{s}(t),t)\delta s(t)\nonumber
  \intertext{in $J_{22}$ to derive}
  J_{22}
    & =\beta_1\int_{T_2} 2\big( u(s(t),t) - \mu(t)\big) \delta u(s(t),t) \,dt +\nonumber
  \\
    & \quad +\beta_1\int_{T_2} 2\big( u(s(t),t) - \mu(t)\big) u_x(s(t),t)\delta s(t) \,dt + o(\delta v).\label{eq:gradient-heuristic-split-J-22}
\end{align}
From~\eqref{eq:gradient-heuristic-split-J-21},~\eqref{eq:gradient-heuristic-split-J-22}, it follows that
\begin{align}
  J_{2}
    & =\beta_1\int_0^T 2\big( u(s(t),t)-\mu(t)\big)u_x(s(t),t)\delta s(t) \,dt + \nonumber
  \\
    & \quad + \beta_1\int_0^T 2\big( u(s(t),t)-\mu(t)\big)\delta u(\shat(t),t) \,dt + o(\delta v).\label{eq:gradient-heuristic-split-J-2-final}
\end{align}

For $J_3$, we have
\begin{align}
  J_3
    & = 2\beta_2(s(T)-s_*)\delta s(T)+o(\delta v).\label{eq:gradient-heuristic-split-J-3-final}
\end{align}
Since $u$ solves PDE~\eqref{eq:pde-1} pointwise almost everywhere, from~\eqref{eq:lagrange-inc-integral} it follows that $\Delta I = 0$ and
\begin{align}
  \Delta I
    & =\int_{\Omegahat} \psi \left[ \big( a {\delta u}_x\big)_x + b {\delta u}_x + c \delta u - {\delta u}_t - \delta f\right] \,dx \,dt.
\end{align}
Integrating by parts with respect to $x$- and $t$-variables, we derive
\begin{align}
  \Delta I & =
  \int_{\Omegahat} \left[(a \psi_x)_x - (\psi b)_x + \psi c + \psi_t\right]\delta u \,dx \,dt
  + \int_0^T \big[a \psi \delta u_x\big]_{x=\shat(t)} \, dt + \nonumber
  \\
           & \quad + \int_0^T \left[\big(-a \psi_x + \big(b + \shat'\big) \psi \big) {\delta u}\right]_{x=\shat(t)} \, dt
  -\int_{\Omegahat} \psi\delta f \,dx \,dt - \nonumber
  \\
           & \quad - \int_0^T \psi(0,t) \delta g(t) \, dt
  + \int_0^T \big[(a \psi_x - b \psi) \delta u\big]_{x=0} \, dt - \nonumber
  \\
           & \quad - \int_0^{\shat(T)} \psi(x,T)\delta u(x,T) \,dx=\nonumber
  \\
           & = \int_{\Omegahat} \left[(a \psi_x)_x - (\psi b)_x + \psi c + \psi_t\right]\delta u \,dx \,dt
  + \int_{T_1} \big[a \psi {\delta u}_x\big]_{x=\sbar(t)} \, dt + \nonumber
  \\
           & \quad + \int_0^T \left[-a \psi_x + \big(b + s'\big) \psi \right]_{x=s(t)} \delta u(\shat(t),t) \, dt
  -\int_{\Omega} \psi\delta f \,dx \,dt + \nonumber
  \\
           & \quad + \int_{T_2} \big[a \psi {\delta u}_x \big]_{x=s(t)} \, dt
  - \int_0^T \psi(0,t) \delta g(t) \, dt +
  \nonumber
  \\
           & \quad + \int_0^T \big[(a \psi_x - b \psi) \delta u\big]_{x=0} \, dt - \int_0^{\shat(T)} \psi(x,T)\delta u(x,T) \,dx + o(\delta v).\label{eq:gradient-heuristic-deltaI-sum-1}
\end{align}
Using the boundary conditions for $\ubar$ on the moving boundary $\sbar$~\eqref{eq:pde-stefan} and mean value theorem, it follows that for $t \in T_1$,
\begin{gather}
  a(\sbar(t),t)\delta u_x(\sbar(t),t) = \left[\chi(\sbar(t),t) - \gamma(\sbar(t),t)\sbar'(t) - a(\sbar(t),t)u_x(\sbar(t),t)\right] - \nonumber
  \\
  -\left[\chi(s(t),t) - \gamma(s(t),t)s'(t) - a(s(t),t)u_x(s(t),t)\right]\nonumber
  \\
  =\chi_x(\tilde{s}(t),t)\delta s(t) - \gamma_x(\tilde{s}(t),t)\delta s(t)\sbar'(t)-\gamma(s(t),t){\delta s}'(t) - (au_x)_x\Big\vert_{x=\tilde{s}(t)}\delta s(t)\label{eq:gradient-heuristic-movingbdy-t1-ident}
  \intertext{Using~\eqref{eq:gradient-heuristic-movingbdy-t1-ident} it follows that}
  \int_{T_1}\left[\psi a \delta u_x\right]_{x=\sbar(t)}\, dt    =\int_{T_1}\psi(\sbar(t),t)\big[\chi_x {\delta s} - \gamma_x {\delta s} \sbar' - (au_x)_x {\delta s}\big]_{x=\tilde{s}(t)}\, dt -\nonumber
  \\
  - \int_{T_1}\psi(\sbar(t),t)\gamma(s(t),t){\delta s}'(t)\, dt\nonumber
  \\
  =\int_{T_1}\Big[\psi \big(\chi_x {\delta s} - \gamma_x {\delta s} s'-\gamma {\delta s}' - (au_x)_x {\delta s}\big)\Big]_{x=s(t)}\, dt + o(\delta v).\label{eq:gradient-heuristic-movingbdy-t1-int-ident}
\end{gather}
Applying the boundary condition~\eqref{eq:pde-stefan} for $u$ on the moving boundary $s$ similarly to the derivation of~\eqref{eq:gradient-heuristic-movingbdy-t1-ident}--\eqref{eq:gradient-heuristic-movingbdy-t1-int-ident}, it follows that, for $t \in T_2$,
\begin{gather}
  \int_{T_2} \left[\psi a \delta u_x \right]_{x=s(t)} \,dt
  =\int_{T_2} \Big[\psi \big(\chi_x {\delta s} - \gamma_x s' {\delta s} - \gamma {\delta s}'\big) - (a u_x)_x \delta s\Big]_{x=s(t)} \,dt +\nonumber
  \\
  + o(\delta v).\label{eq:gradient-heuristic-movingbdy-t2-int-ident}
\end{gather}

Using~\eqref{eq:gradient-heuristic-movingbdy-t1-int-ident} and~\eqref{eq:gradient-heuristic-movingbdy-t2-int-ident} in~\eqref{eq:gradient-heuristic-deltaI-sum-1}, it follows that
\begin{align}
  \Delta I
    & = \int_{\Omegahat} \left[(a \psi_x)_x - (\psi b)_x + \psi c + \psi_t\right]\delta u \,dx \,dt + \nonumber
  \\
    & \quad + \int_0^T \left[-a \psi_x + \big(b + s'\big) \psi \right]_{x=s(t)} \delta u(\shat(t),t) \, dt -\nonumber
  \\
    & \quad -\int_{\Omega} \psi\delta f \,dx \,dt
  + \int_0^T \Big[\psi \big(\chi_x \delta s - \gamma_x s'\delta s - \gamma\delta s' - (a u_x)_x \delta s\big)\Big]_{x=s(t)} \,dt - \nonumber
  \\
    & \quad - \int_{0}^T \psi(0,t) \delta g(t) \, dt
  + \int_0^T \big[(a \psi_x - b \psi) \delta u \big]_{x=0} \, dt - \nonumber
  \\
    & \quad - \int_0^{\shat(T)} \psi(x,T)\delta u(x,T) \,dx + o(\delta v).\label{eq:gradient-heuristic-deltaI-sum-2}
\end{align}

Taking the sum of $\Delta I$ and $\Delta \J$ using~\eqref{eq:functional-increment-1} and~\eqref{eq:gradient-heuristic-split-J-1-final},~\eqref{eq:gradient-heuristic-split-J-2-final},~\eqref{eq:gradient-heuristic-split-J-3-final} gives
\begin{align}
  \Delta \L
    & =
  \int_0^{\shat(T)} \Big[\big(2\beta_0(u-w) - \psi\big) \delta u\Big]_{t=T}\,dx + \nonumber
  \\
    & \quad + \int_0^T \left[-a \psi_x + \big(b + s'\big) \psi + 2\beta_1 (u-\mu)\right]_{x=s(t)} \delta u(\shat(t),t) \,dt + \nonumber
  \\
    & \quad +\int_0^T \big[2\beta_1(u-\mu)u_x \big]_{x=s(t)}\delta s(t) \, dt + \nonumber
  \\
    & \quad + [\beta_0(u(s(T),T)-w(s(T)))^2 + 2 \beta_2(s(T)-s_*)]\delta s(T) + \nonumber
  \\
    & \quad +
  \int_{\Omegahat} \left[(a \psi_x)_x - (\psi b)_x + \psi c + \psi_t\right]\delta u \,dx \,dt + \int_{0}^T \big[(a \psi_x - b \psi) \delta u \big]_{x=0} \, dt + \nonumber
  \\
    & \quad -\int_{\Omega} \psi\delta f \,dx \,dt
  + \int_0^T\Big[ \psi \big(\chi_x\delta s - \gamma_x s' {\delta s} - \gamma \delta s' - (a u_x)_x {\delta s}\big)\Big]_{x=s(t)} \,dt - \nonumber
  \\
    & \quad - \int_{0}^T \psi(0,t) \delta g(t) \, dt + o(\delta v).\label{eq:gradient-heuristic-final-increment}
\end{align}
Due to arbitrariness of the the incremented variables $\delta u$, $\delta s$, etc.\ all of the coefficients on these variables must be zero.
In particular, it follows that $\psi$ should satisfy~\eqref{eq:adj-pde}--\eqref{eq:adj-robin-free} in a pointwise a.e.\ sense.
All of the remaining terms depend linearly on the increment $\delta v$, so the Frechet differential $d\J$ is~\eqref{eq:gradient-full}.
In this form, the adjoint problem~\eqref{eq:adj-pde}--\eqref{eq:adj-robin-free} plays the role of the Lagrange multiplier corresponding to the PDE ``constraint'' in this setting.
\section{Preliminary Results}\label{PreliminaryResults}

\subsection{Existence and Uniqueness of $B_{p,x,t}^{2\l, \l}(Q_T)$-Solutions and Energy Estimates}\label{prel1}

Consider the problem
\begin{gather}
  a u_{xx} + b u_{x} + c u - u_{t}=f~\text{in}~Q_{T}\label{eq:solonnikov-pde}
  \\
  a(0,t)u_{x}(0,t)=\chi_{1}(t),~0 \leq t \leq T
  \\
  a(1,t) u_{x}(1,t) = \chi_{2}(t),~0 \leq t \leq T
  \\
  u(x,0)=\phi(x),~0 \leq x \leq 1.\label{eq:solonnikov-initial}
\end{gather}
Let $\l>1$ be fixed, $p>1$. The following key result is due to Solonnikov~\cite{solonnikov64}
\begin{lem}\label{thm:solonnikov-solution-existence}~\cite[\S7,~thm.\ 17]{solonnikov64}
  Suppose that
  \begin{gather}
    a,b,c \in C_{x,t}^{2\l^*-2,\l^*-1}(Q_T),~\text{arbitrary}~\l^*>\l
    \\
    f \in B_{p,x,t}^{2\l-2,\l-1}(Q_T),\quad
    \phi \in B_p^{2\l-\frac{2}{p}}(0,1),\quad
    \chi_1,\chi_2 \in B_{p}^{\l-\frac{1}{2}-\frac{1}{2p}}(0,T)
  \end{gather}
  and the consistency condition of order
  $k = \left[ \l-\frac{3}{2p}-\frac{1}{2}\right]$
  holds; that is,
  \[
    \D{^j (au_{x})}{x^j}(0,0) = \dD{^j \chi_{1}}{t^j}(0),
    \qquad \D{^j (au_{x})}{x^j}(1,0) = \dD{^j \chi_{2}}{t^j}(0),\quad j=0,\ldots,k.
  \]
  Then the solution $u$ of~\eqref{eq:solonnikov-pde}--\eqref{eq:solonnikov-initial} satisfies the energy estimate
  \begin{gather}
    \norm{u}_{B_{p,x,t}^{2\l,\l}(Q_T)} \leq C \Big[
      \norm{f}_{B_{p,x,t}^{2\l-2,\l-1}(Q_T)}
      + \norm{\phi}_{B_p^{2\l-2/p}(0,1)}
      + \norm{\chi_1}_{B_{p}^{\l-\frac{1}{2}-\frac{1}{2p}}(0,T)} + \nonumber
      \\
      + \norm{\chi_2}_{B_{p}^{\l-\frac{1}{2}-\frac{1}{2p}}(0,T)}
    \Big]\label{eq:solonnikov-energy-est}
    \intertext{when $\l, \l-\frac{3}{2p} \not\in \ZZ_+$, and when $\l \in \ZZ_+$,}
    \norm{u}_{W_{p,x,t}^{2\l,\l}(Q_T)} \leq C \Big[
      \norm{f}_{W_{p,x,t}^{2\l-2,\l-1}(Q_T)}
      + \norm{\phi}_{B_p^{2\l-2/p}(0,1)}
      + \norm{\chi_1}_{B_{p}^{\l-\frac{1}{2}-\frac{1}{2p}}(0,T)} + \nonumber
      \\
    + \norm{\chi_2}_{B_{p}^{\l-\frac{1}{2}-\frac{1}{2p}}(0,T)}\Big].\label{eq:solonnikov-energy-est1}
  \end{gather}
\end{lem}
In particular, energy estimates~\eqref{eq:solonnikov-energy-est},~\eqref{eq:solonnikov-energy-est1} imply the existence and uniqueness of the solution in respective spaces $B_{p,x,t}^{2\l,\l}(Q_T)$ or $W_{p,x,t}^{2\l,\l}(Q_T)$. Note that when $k=0$, the consistency condition of order $k$ is the condition of continuity of the boundary functions:
\begin{gather*}
  (au_x)(0,0) = a(0,0)\phi'(0) = \chi_1(0),\quad (au_x)(1,0) = a(1,0)\phi'(1) = \chi_2(0).
\end{gather*}

\subsection{Traces and Embeddings of Besov Functions}\label{traceembedding}
For functions $u \in W_2^{2,1}(\Omega)$, the applicability of the boundary
conditions are justified by the following trace and regularity
results. From~\cite[lem. II.3.3]{ladyzhenskaya68}, recall
\begin{lem}\label{lem:w221embedding}
If $u \in W_2^{2,1}(\Omega)$, then $u$ has a H\"older continuous representative in $\Omega$; in particular, $u \in C_{x,t}^{1/2,1/4}(\overline{\Omega})$. Moreover~\cite[lem. II.3.4]{ladyzhenskaya68},
the following bounded embeddings of traces hold:
\begin{gather*}
  u\big(s(t),t\big),~u(0,t) \in B_2^{3/4}(0,T),\quad
  u_x\big(s(t),t\big),~u_x(0,t) \in B_2^{1/4}(0,T)
  \intertext{and for any fixed $0 \leq \bar{t} \leq T$,}
  u(\cdot,\bar{t}) \in W_2^1(0,s(\bar{t})).
\end{gather*}
\end{lem}
From~\cite[\S 4, thm.\ 9]{solonnikov64}, recall
\begin{lem}\label{thm:solonnikov-traces}
  For a function $u \in B_2^{2\l,\l}(Q_T)$, the following bounded embeddings of traces hold: for any fixed $0 \leq t \leq T$,
  \begin{gather*}
    u(\cdot,t) \in B_2^{2\l-1}[0,1]~\text{when}~\l > 1/2.
    \intertext{For any fixed $0\leq x \leq 1$,}
    u(x,\cdot) \in B_2^{\l-1/4}[0,T]~\text{when}~\l > 1/4
    \\
    u_x(x,\cdot) \in B_2^{\l-3/4}[0,T]~\text{when}~\l > 3/4
    \\
    u_{xx}(x,\cdot) \in B_2^{\l-5/4}[0,T]~\text{when}~\l > 5/4.
  \end{gather*}
\end{lem}

\subsection{Consequences of Energy Estimates and Embeddings}\label{sec:conseq-of-energy-est}
For given $v=(f,s,g) \in V_R$ transform the domain $\Omega$ to the cylindrical domain $Q_T$
by the change of variables $y = x / s(t)$. Let $d = d(x, t)$, $(x, t) \in \Omega$ stand for any of $a,b,c,f,\gamma,\chi$, define the function $\tilde{d}$ by
\begin{gather*}
  \tilde{d}(x,t) = d\big(x s(t), t\big),~\text{and}~
  \tilde{\phi}(x) = \phi\big( x s(t)\big).
\end{gather*}
The transformed function $\utilde$ is a \emph{pointwise a.e.} solution of the Neumann problem
\begin{gather}
  \frac{1}{s^2}\big(\tilde{a} \utilde_y\big)_y + \frac{1}{s}\big(\tilde{b} + y s'(t)\big) \utilde_y + \tilde{c} \utilde - \utilde_{t} = \tilde{f}, ~\text{in}~Q_T\label{eq:pde-flat}
  \\
  \utilde(x,0) = \tilde{\phi}(x), ~0 \leq x \leq 1\label{eq:pde-init-flat}
  \\
  \tilde{a}(0, t) \utilde_y(0, t) = g(t)s(t), ~0 \leq t \leq T\label{eq:pde-bound-flat}
  \\
  \tilde{a}(1, t) \utilde_y(1, t) = \tilde{\chi}(1, t) s(t) - \tilde{\gamma}(1,t)s'(t)s(t), ~0 \leq t \leq T.\label{eq:pde-free-flat}
\end{gather}
\begin{lem}\label{lem:j-well-defined}
  For fixed $v \in V_R$, there exists a unique solution $u \in W_2^{2,1}(\Omega)$ of the Neumann problem~\eqref{eq:pde-1}--\eqref{eq:pde-stefan} for which the transformed function $\utilde \in \solnspace(Q_T)$ solves~\eqref{eq:pde-flat}--\eqref{eq:pde-free-flat} and satisfies the following energy estimate
  \begin{gather}
    \norm{\utilde}_{\solnspace(Q_T)} \leq C \Big(\norm{f}_{\fcontrolspace(D)} + \norm{\phi}_{\ivspace(0,s_0)} +\nonumber
    \\
    +\norm{g}_{\uxxtrspace[0,T]}
    +\norm{\chi}_{\chigammaspace(D)}
    +\norm{\gamma}_{\chigammaspace(D)}\Big).\label{eq:energy-est-utilde}
  \end{gather}
  where $\alpha^*>\alpha$ is arbitrary.
\end{lem}
\begin{proof}
  Assumptions~\eqref{eq:a-ellipticity-cond}--\eqref{eq:datacond-compat} imply the applicability of Lemma~\ref{thm:solonnikov-solution-existence} with $p = 2$, $\l = 5/4 + \alpha$ to the Neumann problem~\eqref{eq:pde-flat}--\eqref{eq:pde-free-flat}; by energy estimate~\eqref{eq:solonnikov-energy-est},
  \begin{gather}
    \norm{\utilde}_{\solnspace(Q_T)}
    \leq C\Big[
      \norm{\tilde{f}}_{\densityspace(Q_T)}
      + \norm{\tilde{\phi}}_{\ivspace(0,1)} +\nonumber
      \\
      + \norm{g(t)s(t)}_{\uxxtrspace(0,T)} +
      + \norm{\tilde{\chi}(1, t) s(t) - \tilde{\gamma}(1,t)s'(t)s(t)}_{\uxxtrspace(0,T)}
    \Big].\nonumber
  \end{gather}
  By Sobolev trace embedding recalled in Lemma~\ref{thm:solonnikov-traces} 
  we have
  \begin{gather*}
    \norm{\tilde{\chi}(1, t) s(t) - \tilde{\gamma}(1,t)s'(t)s(t)}_{\uxxtrspace(0,T)} \leq C \Big( \norm{\tilde{\chi}}_{\chigammaspace(Q_T)}+
    \\
    +\norm{\tilde{\gamma}}_{\chigammaspace(Q_T)}\Big).
  \end{gather*}
  Considering the term $\norm{\tilde{f}}_{\densityspace(Q_T)}$, by definition
  \begin{gather*}
    \norm{\tilde{f}}_{\densityspace(Q_T)}
    = \norm{\tilde{f}}_{L_2(Q_T)} + \nonumber
    \\
    + \left( \int_0^T \int_0^1 \int_0^1
    \frac{
      \lnorm{f(xs(t),t) - f(ys(t),t)}^2
      }{
      \lnorm{x-y}^{1+2(1/2+\alpha)}
    }
    \, dy \,dx \,d\tau\right)^{1/2} + \nonumber
    \\
    + \left( \int_0^T \int_0^T \int_0^1
    \frac{
      \lnorm{f(ys(t),t) - f(ys(\tau),\tau)}^2
      }{
      \lnorm{t-\tau}^{1+2(1/4+\alpha)}
    }
    \, dy \,dt \,d\tau\right)^{1/2}.\nonumber
  \end{gather*}
  The first two terms are estimated in a straightforward way in terms of the norm of $f$; the last term can be estimated through the sum of two terms $I_1 + I_2$ where
  \begin{align}
    I_1
      & =\left( \int_0^T \int_0^T \int_0^1
    \frac{
    \lnorm{f(ys(\tau),t) - f(ys(\tau),\tau)}^2
    }{
    \lnorm{t-\tau}^{1+2(1/4+\alpha)}
    }
    \, dy \,dt \,d\tau\right)^{1/2},\nonumber
    \\
    I_2
      & =\left( \int_0^T \int_0^T \int_0^1
    \frac{
    \lnorm{f(ys(t),t) - f(ys(\tau),t)}^2
    }{
    \lnorm{t-\tau}^{1+2(1/4+\alpha)}
    }
    \, dy \,dt \,d\tau\right)^{1/2}.\nonumber
    \intertext{$I_1$ can also be estimated in terms of $\norm{f}_{\densityspace(D)}$; the last requires increased regularity of $f$ with respect to the space variable: by CBS inequality,}
    I_2^2
      & \leq \int_0^T \int_0^T \int_0^1 \int_0^1
    \frac{
    \lnorm{f_x\big(y(\theta s(t) + (1-\theta) s(\tau)),t\big)}^2\lnorm{s(t)-s(\tau)}^2
    }{
    \lnorm{t-\tau}^{1+2(1/4+\alpha)}
    }
    \,d\theta \, dy \,dt \,d\tau.\nonumber
    \intertext{By mean value theorem, Morrey's inequality, and condition~\eqref{eq:control-set} on $s$,}
    I_2
      & \leq C\norm{s'}_{W_2^1[0,T]} T^{3/4-\alpha}\frac{1}{\sqrt{\delta}} \norm{f_x}_{L_2(D)}.\nonumber
  \end{align}
  Similarly,
  \begin{gather*}
    \norm{\tilde{\chi}}_{\chigammaspace(Q_T)} \leq
    C\norm{\chi}_{\chigammaspace(D)}
    \\
    \norm{\tilde{\gamma}}_{\chigammaspace(Q_T)} \leq
    C \norm{\gamma}_{\chigammaspace(D)}.
  \end{gather*}
  Estimate~\eqref{eq:energy-est-utilde} follows.
\end{proof}
By the trace embedding result of Lemma~\ref{lem:w221embedding}, it follows that the functional $\J(v)$ of~\eqref{eq:functional} is well defined for $v \in V_R$.
\begin{lem}\label{lem:adjoint-solution-exists}
  For fixed $v \in V_R$, given the corresponding state vector $u=u(x,t;v)$ there exists a unique solution $\psi \in \adjointsolnspace(\Omega)$ of the adjoint problem~\eqref{eq:adj-pde}--\eqref{eq:adj-robin-free} and the following energy estimate is valid
  \begin{equation}\label{eq:psi-energy-est}
    \begin{gathered}
      \norm{\psi}_{\adjointsolnspace(\Omega)} \leq C \Big(\norm{f}_{L_2(\Omega)}
      + \norm{\phi}_{W_2^1(0,s_0)}
      + \norm{g}_{\uxxtrspace[0,T]} +
      \\
      +\norm{\chi}_{B_{2,x,t}^{1+2\alpha,\frac{1}{2}+\alpha}(D)}
      +\norm{s}_{W_2^2[0,T]} \norm{\gamma}_{B_{2,x,t}^{1+2\alpha,\frac{1}{2}+\alpha}(D)}
      + \norm{w}_{W_2^1[0,s(T)]} +
      \\
      + \norm{\mu}_{\muspace(0,T)}\Big).
    \end{gathered}
  \end{equation}
\end{lem}
\begin{proof}
  By Lemma~\ref{lem:j-well-defined}, there exists a unique solution $u$ of the problem~\eqref{eq:pde-1}--\eqref{eq:pde-stefan}. In particular, $W_2^{2,1}$-norm of the solution satisfies the following energy estimate
  \begin{equation}\label{eq:w221-solution-energy-est}
    \begin{gathered}
      \norm{u}_{W_2^{2,1}(\Omega)} \leq C \Big( \norm{f}_{L_2(\Omega)} + \norm{\phi}_{W_2^1[0,s_0]}
      +\norm{g}_{B_2^{1/4}(0,T)} +
      \\
      + \norm{\chi |_{x=s(t)}-\gamma |_{x=s(t)}s'(t)}_{B_2^{1/4}(0,T)}\Big).
    \end{gathered}
  \end{equation}
  Indeed, by applying~\eqref{eq:solonnikov-energy-est1} to the transformed solution $\utilde$, and due to equivalence of $W_2^1[0,\l]$ and $B_2^1[0,\l]$ norms,~\eqref{eq:w221-solution-energy-est} follows. From the trace embedding result of Lemma~\ref{thm:solonnikov-traces} it follows that
  \begin{equation*}
    \begin{gathered}
      \norm{\chi(s(t),t)}_{B_2^{1/4}(0,T)}+\norm{\gamma(s(t),t)}_{B_2^{1/4}(0,T)}\leq C \Big(\norm{\chi}_{B_{2,x,t}^{1+2\alpha,\frac{1}{2}+\alpha}(D)}+
      \\
      +\norm{\gamma}_{B_{2,x,t}^{1+2\alpha,\frac{1}{2}+\alpha}(D)}\Big).
    \end{gathered}
  \end{equation*}
  and therefore from~\eqref{eq:w221-solution-energy-est} we get
  \begin{gather}
    \norm{u}_{W_2^{2,1}(\Omega)}
    \leq C \bigg( \norm{f}_{L_2(\Omega)} + \norm{\phi}_{W_2^1[0,s_0]} + \norm{g}_{B_2^{1/4}(0,T)} + \norm{\chi}_{B_{2,x,t}^{1+2\alpha,\frac{1}{2}+\alpha}(D)}+ \nonumber
    \\
    +\norm{s}_{W_2^2[0,T]} \norm{\gamma}_{B_{2,x,t}^{1+2\alpha,\frac{1}{2}+\alpha}(D)}\bigg).\label{eq:w221-solution-energy-est-1}
  \end{gather}
  From~\cite[thm. IV.9.1]{ladyzhenskaya68} and~\cite[thm. 5.4]{solonnikov65} it follows that there exists a unique solution $\psi\in \adjointsolnspace(\Omega)$ of the third type adjoint problem~\eqref{eq:adj-pde}--\eqref{eq:adj-robin-free} and the following energy estimate is satisfied:
  \begin{gather}
    \norm{\psi}_{\adjointsolnspace(\Omega)}
    \leq C \left( \norm{u(x,T;v)-w(x)}_{W_2^1[0,s(T)]} + \norm{u(s(t),t)-\mu(t)}_{B_2^{1/4}(0,T)}\right.)\nonumber
    \intertext{By embedding theorem recalled in
      Lemma~\ref{lem:w221embedding}~\cite[lem. II.3.3]{ladyzhenskaya68} it
      follows that}
    \norm{\psi}_{\adjointsolnspace(\Omega)}
\leq C \left( \norm{u}_{\adjointsolnspace(\Omega)} + \norm{w}_{W_2^1[0,s(T)]} + \norm{\mu}_{\muspace(0,T)}\right).\nonumber
  \end{gather}
  Applying energy estimate~\eqref{eq:w221-solution-energy-est-1},~\eqref{eq:psi-energy-est} follows. Lemma is proved.
\end{proof}
\def\ybar{\bar{y}}
The following technical Lemma plays a key role in showing that remainder terms appearing in the derivation of formula~\eqref{eq:gradient-heuristic-final-increment} for the increment are of higher than linear order with respect to $\delta v$, and hence in completing the rigorous proof of the main result.
Take $\utilde$, $\tilde{a}$, etc.\ as in the proof of Lemma~\eqref{lem:j-well-defined}, transform the domain $\Omega_{\bar{s}}$ to $Q_T$ by taking
$\ybar = x / \sbar(t)$, etc.\ in a similar way, and define $\ubartilde$, $\tildebar{a}$, etc.\@ For $d$ standing for any of $u$, $f$, $a$, $b$, $c$, $\gamma$, $\chi$, denote by
\[
  \Delta \tilde{d}(y,t) = \tildebar{d}(y,t)-\tilde{d}(y,t),~(y,t) \in Q_T.
\]
\begin{lem}\label{lem:deltau-main-est}
If $a,b,c$ satisfy~\eqref{eq:a-ellipticity-cond},~\eqref{eq:datacond-coeff} and $\chi,\gamma$ satisfy~\eqref{eq:datacond-4}, then
\begin{gather}
  \norm{\Delta \utilde}_{\solnspace(Q_T)} \to 0~\text{as}~ \delta v \to 0~\text{in}~H.\label{eq:deltau-vanishes}
\end{gather}
\end{lem}
\begin{proof}
  A straightforward calculation shows that $\Delta \utilde$ solves
  \begin{gather}
    \frac{\tildebar{a}}{\bar{s}^2} \Delta \utilde_{yy} + \frac{1}{\bar{s}}\left( \frac{\tildebar{a}_y}{\bar{s}} + \tildebar{b} + y \bar{s}'\right)\Delta \utilde_y + \tildebar{c} \Delta \utilde - \Delta \utilde_t
    = \Delta \tilde{f} +
    \left(\frac{\tilde{a}}{s^2}-\frac{\tildebar{a}}{\bar{s}^2}\right) \utilde_{yy} +\nonumber
    \\
    + \left( \frac{\tilde{a}_y}{s^2} - \frac{\tildebar{a}_y}{\bar{s}^2} + \frac{\tilde{b}}{s} - \frac{\tildebar{b}}{\bar{s}} + \frac{s'}{s}-\frac{\bar{s}'}{\bar{s}}\right) \utilde_y
    - {\Delta \tilde{c}} \utilde\label{eq:deltautilde-pde}
    \intertext{in $Q_T$,}
    \Delta \utilde(x,0) = 0,~0\leq x \leq 1\label{eq:deltautilde-init}
    \\
    \tildebar{a}(0,t)\Delta \utilde_y(0,t) = \delta s(t) g(t)+s(t)\delta g(t)+\delta s(t) \delta g(t),~0 \leq t \leq T,\label{eq:deltautilde-lbc}
    \intertext{and}
    \tildebar{a}(1,t)\Delta \utilde_y(1,t) = -\Delta \tilde{a}(1,t) \utilde_y(1,t)
    +\left( \tildebar{\chi}(1,t)-\tildebar{\gamma}(1,t)\bar{s}'(t)\right)\delta s(t) + \nonumber
    \\
    +\Big(\Delta \tilde{\chi}(1,t) - \tildebar{\gamma}(1,t)\delta s'(t) - \Delta\tilde{\gamma}(1,t)s'(t)\Big)s(t),\label{eq:deltautilde-rbc}
  \end{gather}
  By the energy estimate~\eqref{eq:solonnikov-energy-est} for functions in $\solnspace$, it follows from equations~\eqref{eq:deltautilde-pde}--\eqref{eq:deltautilde-rbc}, and Minkowski inequality that
  \begin{gather}
    \norm{\Delta \utilde}_{\solnspace(Q_T)}
    \leq C \bigg[
      \sum_{i=1}^6 \norm{\Gamma_i}_{\densityspace(Q_T)}+ \nonumber
      \\
      + \norm{\tildebar{a}(0,t) {\Delta \utilde}_y(0,t)}_{\uxxtrspace[0,T]}
      + \norm{\tildebar{a}(1,t) {\Delta \utilde}_y(1,t)}_{\uxxtrspace[0,T]}
    \bigg],\label{eq:deltautilde-calc-1}
  \end{gather}
  where
  \begin{align}
    \Gamma_1
      & =\Delta \tilde{f},\quad
    \Gamma_2
    =\left(\frac{\tilde{a}}{s^2}-\frac{\tildebar{a}}{\bar{s}^2}\right) \utilde_{yy},\quad
    \Gamma_3
    = \left( \frac{\tilde{a}_y}{s^2} - \frac{\tildebar{a}_y}{\bar{s}^2}\right)\utilde_y\nonumber
    \\
    \Gamma_4
      & =\left(\frac{\tilde{b}}{s} - \frac{\tildebar{b}}{\bar{s}}\right)\utilde_y,\quad
    \Gamma_5
    =\left(\frac{s'}{s}-\frac{\bar{s}'}{\bar{s}}\right) \utilde_y,\quad
    \Gamma_6
    =\Delta \tilde{c}\utilde.\nonumber
  \end{align}
  The last two terms in~\eqref{eq:deltautilde-calc-1} converge to zero by~\eqref{eq:deltautilde-lbc},~\eqref{eq:deltautilde-rbc}. The remaining terms are all estimated in a similar way; we demonstrate the estimation for one such term.
  Since $\sbar \to s$ uniformly as $\delta v \to 0$, to prove that $\norm{\Gamma_2} \to 0$, due to Lebesgue's dominated convergence theorem, it follows that it is sufficient to show that the integrands in
  \[
    \norm{\frac{\tilde{a}}{s^2} \utilde_{yy}}_{\densityspace(Q_T)}, \norm{\frac{\tildebar{a}}{\sbar^2} \utilde_{yy}}_{\densityspace(Q_T)}
  \]
  are bounded above by an integrable function. We demonstrate how the first term is estimated; the second is estimated in a nearly identical way. By definition,
  \begin{gather}
    \norm{\frac{\tilde{a}}{s^2} \utilde_{yy}}_{\densityspace(Q_T)}
    = \left(
    \int_0^T\int_0^1 I_1 \,dy \,dt
    \right)^{1/2} + \nonumber
    \\
    + \left(
    \int_0^T\int_0^1\int_0^1 I_2 \,dy \,dx \,dt
    \right)^{1/2}
    + \left(
    \int_0^T\int_0^T\int_0^1 I_3 \,dy \,dt \,d\tau\right)^{1/2},\label{eq:utildeyy-term}
  \end{gather}
  where
  \begin{align}
    I_1
      & =\lnorm{\frac{1}{s^2(t)}a(ys(t),t)\utilde_{yy}(y,t)}^2,\nonumber
    \\
    I_2
      & =\frac{\lnorm{\frac{a(ys(t),t)}{s^2(t)} \utilde_{yy}(y,t) - \frac{a(xs(t),t)}{s^2(t)} \utilde_{yy}(x,t)}^2
    }{\lnorm{x-y}^{1+2(1/2+2\alpha)}}\nonumber
    \\
    I_3
      & =\frac{\lnorm{\frac{a(ys(t),t)}{s^2(t)} \utilde_{yy}(y,t) - \frac{a(ys(\tau),\tau)}{s^2(\tau)} \utilde_{yy}(y,\tau)}^2
    }{\lnorm{t-\tau}^{1+2(1/4+\alpha)}},\nonumber
  \end{align}
  Estimate in $I_1$ using the condition~\eqref{eq:control-set} on $s$ and condition~\eqref{eq:datacond-coeff} on $a$, and~\eqref{eq:energy-est-utilde} to derive
  \begin{align}
    I_1
      & \leq \frac{1}{\delta^2}\norm{a}_{C(D)}^2\lnorm{\utilde_{yy}(y,t)}^2,
  \end{align}
  By~\eqref{eq:control-set},~\eqref{eq:datacond-coeff},
  \begin{align}
    I_2
      & \leq \frac{2}{\delta^2}\Big(
    \l^2\norm{a_x}_{C(D)}^2
    \lnorm{\utilde_{yy}(y,t)}^2
    \lnorm{x-y}^{-4\alpha}+\nonumber
    \\
      & \quad+\norm{a}_{C(D)}^2
    \frac{\lnorm{\utilde_{yy}(y,t) - \utilde_{yy}(x,t)}^2
    }{\lnorm{x-y}^{2+4\alpha}}\Big),
  \end{align}
  Both $I_1$ and $I_2$ are integrable by definition of the norm in $\solnspace(Q_T)$. In a similar way, derive
  \begin{align}
    I_3 & \leq 4 \left(I_{31} + I_{32} + I_{33} + I_{34} \right),\nonumber
    \intertext{where}
    I_{31}
        & =\frac{\lnorm{a(ys(t),t)-a(ys(\tau),t)}^2 \lnorm{\utilde_{yy}(y,t)}^2
    }{\lnorm{t-\tau}^{3/2+2\alpha}s^4(t)},\nonumber
    \\
    I_{32}
        & =
    \frac{\lnorm{a(ys(\tau),t) - a(ys(\tau),\tau)}^2 \lnorm{\utilde_{yy}(y,t)}^2
    }{\lnorm{t-\tau}^{3/2+2\alpha} s^4(t)}\nonumber
    \\
    I_{33}
        & =
    \frac{\lnorm{a(ys(\tau),\tau)}^2 \lnorm{s^2(\tau)-s^2(t)}^2 \lnorm{\utilde_{yy}(y,t)}^2
    }{\lnorm{t-\tau}^{3/2+2\alpha}\lnorm{s^2(t) s^2(\tau)}^2},\nonumber
    \\
    I_{34}
        & =
    \frac{\lnorm{a(ys(\tau),\tau)}^2 \lnorm{\utilde_{yy}(y,t) - \utilde_{yy}(y,\tau)}^2
    }{\lnorm{t-\tau}^{3/2+2\alpha}s^4(\tau)}.\nonumber
  \end{align}
  By estimate~\eqref{eq:control-set} on $s$, assumption~\eqref{eq:datacond-coeff}, mean value theorem and Morrey inequality,
  \begin{gather}
    I_{31} + I_{32} + I_{33} \leq \bigg(
    \frac{C \norm{s'}_{W_2^1[0,T]}^2 T^{1/2-2\alpha}\norm{a_x}_{C(D)}^2}{\delta^4}
    +\frac{\norm{a}_{C_{x,t}^{0,1/4+\alpha^*}(D)}^2}{
      \lnorm{t-\tau}^{1-2(\alpha^*-\alpha)}\delta^4}\nonumber
    \\
    +C \frac{\norm{a}_{C(D)}^2 4 \l^2 \norm{s'}_{W_2^1[0,T]}^2 T^{1/2-2\alpha}}{\delta^4}\bigg)\lnorm{\utilde_{yy}(y,t)}^2,\nonumber
  \end{gather}
  which is integrable by definition of norm in $\solnspace(Q_T)$ and integrability of the singularities on the curve $t=\tau$ in $\R^2$.

  The term $I_{34}$ requires increased time regularity of $\utilde_{yy}$, which, for arbitrary functions in $\solnspace(Q_T)$, does not follow from the definition of the norm. However, since $\utilde$ is a pointwise a.e.\ solution of~\eqref{eq:pde-flat}, estimate~\eqref{eq:control-set} on $s$, and assumption~\eqref{eq:datacond-coeff} ($a \in C(D)$), and triangle inequality, it follows that
  \begin{align}
    I_{34}
    & \leq \frac{4\norm{a}_{C(D)}^2}{\delta^4}\left(I_{34}^1 + I_{34}^2 + I_{34}^3 + I_{34}^4\right)\nonumber
   \end{align}
        where
   \begin{align}
    I_{34}^1
      & =
    \frac{ \lnorm{\left[\frac{\tilde{a}_x + s\tilde{b} + y s' s}{\tilde{a}} \utilde_y\right]_{\tau}^{t}}^2
    }{\lnorm{t-\tau}^{3/2+2\alpha}}\nonumber
    \\
    I_{34}^2
      & =
    \frac{ \lnorm{a(x s(\tau),\tau)s^2(t)\tilde{c}(x,t) \utilde(x,t) - a(x s(t),t)s^2(\tau)\tilde{c}(x,\tau) \utilde(x,\tau)}^2
    }{\lnorm{t-\tau}^{3/2+2\alpha}\lnorm{a(x s(\tau),\tau)a(x s(t),t)}^2}\nonumber
    \\
    I_{34}^3
      & =
    \frac{ \lnorm{a(x s(\tau),\tau)s^2(t)\utilde_{t}(x,t) - a(x s(t),t)s^2(\tau)\utilde_t(x,\tau)}^2
    }{\lnorm{t-\tau}^{3/2+2\alpha} \lnorm{a(x s(\tau),\tau)a(x s(t),t)}^2}\nonumber
    \\
    I_{34}^4
      & =
    \frac{ \lnorm{a(x s(\tau),\tau)s^2(t)\tilde{f}(x,t) - a(x s(t),t)s^2(\tau)\tilde{f}(x,\tau)}^2
    }{\lnorm{t-\tau}^{3/2+2\alpha}\lnorm{a(x s(\tau),\tau)a(x s(t),t)}^2}.\nonumber
  \end{align}
  Each term is now easily estimated using Sobolev embedding;
  we demonstrate the estimation of the highest regularity term $I_{34}^3$.
  By triangle inequality, condition~\eqref{eq:a-ellipticity-cond} on $a$ and condition~\eqref{eq:control-set} on $s$, assumption~\eqref{eq:datacond-coeff} ($a_x \in C(D)$ and $a \in C_{x,t}^{0,1/4+\alpha^*}(D)$,) mean value theorem, and Morrey's inequality,
  \begin{align}
    I_{34}^3
    & \leq \frac{\l^4}{a_0^4}\norm{a_x}_{C(D)}^2\norm{s'}_{W_2^1[0,T]}^2
    \lnorm{\utilde_{t}(y,t)}^2 T^{1/2-2\alpha} +\nonumber
    \\
    &\quad
    +C\frac{\l^4}{a_0^4}\norm{a}_{C_{x,t}^{0,1/4+\alpha}(D)}^2
    \frac{\lnorm{\utilde_{t}(x,t)}^2
    }{\lnorm{t-\tau}^{1-2(\alpha^*-\alpha)}} + \nonumber
    \\
    & \quad + \frac{4\l^2}{a_0^2}\norm{s'}_{W_2^1[0,T]}^2
    T^{1/2-2\alpha} \lnorm{\utilde_{t}(y,t)}^2 + \frac{\l^4}{a_0^2}\frac{ \lnorm{\utilde_{t}(x,t) - \utilde_t(x,\tau)}^2
    }{\lnorm{t-\tau}^{3/2+2\alpha}},
  \end{align}
  which is integrable by the definition of the norm in $\solnspace(Q_T)$.
\end{proof}

\section{Proofs of Main Results}\label{sec:gradient-rigorous}
{\it Proof of Theorem~\ref{thm:existence-opt-control}}. 
Let $\bk{v_n=(s_n,g_n,f_n)} \subseteq V_R$ be a minimizing sequence, i.e.\
\[
  \J(v_n) \to \J_*.
\]
Since $V_R$ is a bounded and closed subset of an Hilbert space $H$, $\bk{v_n}$ is weakly precompact (see e.g.~\cite[ch.\ V, \S\ 2, p.126]{yosida96}); that is, there exists a subsequence $v_{n_k}$ which converges weakly. Assume the whole sequence converges weakly to $v_* = (s_*,g_*,f_*) \in V_R$.
Compact Sobolev embedding theorem implies $v_n \to v$ strongly in $W_2^1[0,T]\times L_2[0,T]\times L_2(D)$. The corresponding solutions $u_n \in W_2^{2,1}(\Omega)$ and transformed solutions $\utilde_n \in W_2^{2,1}(Q_T)$ are uniformly bounded by~\eqref{eq:w221-solution-energy-est}; that is, there exists $C > 0$ such that
\begin{gather*}
  \norm{\utilde_n}_{W_2^{2,1}(Q_T)} \leq C.
  \intertext{It follows that there exists an element $\tilde{v} \in W_2^{2,1}(Q_T)$ such that for some subsequence $n_k$,}
  \utilde_{n_k} \to \tilde{v}~\text{weakly in}~W_2^{2,1}(Q_T).
\end{gather*}
Multiplying~\eqref{eq:pde-1} written for $\utilde_{n_k}$ by an arbitrary test function $\Phi \in L_2(Q_T)$ and passing to the limit as $n_k\to \infty$, it follows that $\tilde{v}$ is a $W_2^{2,1}(Q_T)$ weak solution of~\eqref{eq:pde-flat}--\eqref{eq:pde-free-flat}.
Since all weak limit points of $\bk{\utilde_n}$ are $W_2^{2,1}(Q_T)$ weak solutions of the same equation, by uniqueness of the weak solution, it follows that $\tilde{v} = \utilde_*$ and the whole sequence $\utilde_n \to \utilde_*$ weakly in $W_2^{2,1}(Q_T)$. Sobolev trace theorem~\cite{besov79} then implies that
\begin{gather*}
  \utilde_n(y,T) \to \utilde_*(y,T)~\text{in}~C[0,1],
  \intertext{and hence}
  u_n(y s_n(T),T) \to u_*(y s_*(T),T)~\text{in}~C[0,1].
  \intertext{Together with the convergence of $s_n(t) \to s_*(t)$ uniformly on $[0,T]$, it follows that}
  \int_0^{s_n(T)}\lnorm{u_n(x,T)-w(x)}^2 \, dx - \int_0^{s_*(T)}\lnorm{u_*(x,T)-w(x)}^2 \, dx \to 0.
  \intertext{The other two terms in $\J(v)$ are handled similarly, so it follows that}
  \lim_{n\to\infty} \J(v_n) = \J(v_*) = \J_*.
\end{gather*}
Theorem~\ref{thm:existence-opt-control} is proved.

We lastly give a rigorous justification for the formula given for the first variation of the functional $\J$ in Section~\ref{sec:gradient-heuristic}; repeated calculations will now be omitted with a reference to the previous equation.

The important facts now used are
\begin{itemize}
  \item Under the assumptions~\eqref{eq:a-ellipticity-cond}--\eqref{eq:datacond-compat}, and given $v \in V_R$, there exists a unique solution $u \in \solnspace$ of~\eqref{eq:pde-1}--\eqref{eq:pde-stefan} (Lemma~\ref{lem:j-well-defined}),
  \item Subsequently, the corresponding adjoint $\psi \in \adjointsolnspace$ (Lemma~\ref{lem:adjoint-solution-exists}) also exists.
  \item Then, precise embedding results (Lemmas~\ref{lem:w221embedding} and~\ref{thm:solonnikov-traces}) apply to guarantee that all of the traces appearing in the derivation exist as well.
  \end{itemize}
{\it Proof of Theorem~\ref{thm:gradient-result2}}. 
Proceed just as in Section~\ref{sec:gradient-heuristic}, partitioning the time domain as $[0,T]=T_{1}\cup T_{2}$ as before, and consider first the increment of the functional, $\Delta \J$. Decomposing the increment as in~\eqref{eq:functional-increment-1}, formula~\eqref{eq:gradient-heuristic-split-J-1-final} is valid with $o(\delta v)$ replaced by $R_1 + \cdots + R_5$ where
\begin{align}
  R_1
      & =\Ind_{T_1}(T)\beta_0\int_0^{\sbar(T)} \lnorm{\delta u(x,T)}^2 \, dx,\quad R_2 =\Ind_{T_2}(T)\beta_0\int_0^{s(T)} \lnorm{\delta u(x,T)}^2 \, dx,\nonumber
  \\
  R_3
      & = \Ind_{T_1}(T)\beta_0\left(\lnorm{u(\tilde{s}(T),T) - w(\tilde{s}(T))}^2 - \lnorm{u(s(T),T) - w(s(T))}^2 \right)\delta s(T),\nonumber
  \\
  R_4
  & = \Ind_{T_2}(T)\beta_0\left(\lnorm{\ubar(\tilde{s}(T),T) - w(\tilde{s}(T))}^2 - \lnorm{\ubar(s(T),T) - w(s(T))}^2 \right)\delta s(T),\nonumber
  \\
  R_5
  & = \Ind_{T_2}(T)\beta_0\left(\lnorm{\ubar(s(T),T) - w(s(T))}^2 - \lnorm{u(s(T),T) - w(s(T))}^2\right)\delta s(T).\nonumber
\end{align}
Similarly, formula~\eqref{eq:gradient-heuristic-split-J-2-final} is valid with $o(\delta v)$ replaced by $R_6 + \cdots + R_{10}$ where
\begin{align*}
  R_6
    & = \int_{T_1}\beta_1\lnorm{u_x(\tilde{s}(t),t)\delta s(t) + \delta u(\sbar(t),t)}^2 \,dt,\nonumber
  \\
  R_7
    & =\int_{T_2} \beta_1 \lnorm{\delta u(s(t),t) + \ubar_x(\tilde{s}(t),t)\delta s(t)}^2 \,dt,\nonumber
  \\
  R_8
    & = \int_{T_1}2\beta_1\big( u-\mu\big)_{x=s(t)}\left[u_x\right]^{x=\tilde{s}(t)}_{x=s(t)}\delta s(t) \,dt,\nonumber
  \\
  R_9
    & =\int_{T_2} 2\beta_1\big(u - \mu\big)_{x=s(t)} \left[\ubar_x\right]_{x=s(t)}^{x=\tilde{s}(t)}\delta s(t) \,dt,\nonumber
  \\
  R_{10}
    & =\int_{T_2} 2\beta_1\big( u(s(t),t) - \mu(t)\big) \left(\ubar_x(s(t),t)-u_x(s(t),t)\right)\delta s(t) \,dt.\nonumber
\end{align*}
Lastly,
\begin{align}
  J_3
    & = 2\beta_2(s(T)-s_*)\delta s(T)+R_{11}~\text{where}~
  R_{11}=\beta_2 \lnorm{\delta s(T)}^2.\label{eq:rigorous-split-J-3-final}
\end{align}
Proceeding as before with the term $\Delta I$, formula~\eqref{eq:gradient-heuristic-deltaI-sum-1} is valid with $o(\delta v)$ replaced by $R_{12} + R_{13}$ where
\begin{align}
  R_{12}
    & =\int_{T_1} \int_{\sbar(t)}^{s(t)} \psi \delta f \,dx \,dt,\quad
  R_{13}
  =\int_{T_1} \Big(-a \psi_x + b \psi\Big)\Big\vert_{x=s(t)}^{x=\sbar(t)} \delta u(\sbar(t),t) \, dt.\nonumber
\end{align}
Using the boundary condition~\eqref{eq:pde-stefan} for $\ubar$ on the moving boundary $\sbar$, as in~\eqref{eq:gradient-heuristic-movingbdy-t1-ident}, \eqref{eq:gradient-heuristic-movingbdy-t1-int-ident} it follows that
\begin{gather}
  \int_{T_1} \psi a {\delta u}_x\Big\vert_{x=\sbar(t)}\, dt               =\int_{T_1}\Big[\psi \big(\chi_x {\delta s} - \gamma_x \delta s \sbar'-\gamma {\delta s}'
  - (au_x)_x{\delta s} \big)\Big]_{x=s(t)}\, dt +\nonumber
  \\
  + \sum_{i=14}^{18} R_i,\label{eq:rigorous-movingbdy-t1-int-ident}
  \intertext{where}
  R_{14}
  =\int_{T_1}\left[\psi\right]_{x=s(t)}^{x=\sbar(t)}\Big[\chi_x {\delta s} - \gamma_x {\delta s} s' -\gamma {\delta s}'\Big]_{x=s(t)} \,dt,\nonumber
  \\
  R_{15}
  =\int_{T_1} \left[\psi\right]_{x=\sbar(t)}^{x=s(t)} \Big[(au_x)_x\Big]_{x=\tilde{s}(t)}\delta s(t)\, dt,\nonumber
  \\
  R_{16}
  =\int_{T_1}\psi(s(t),t)\Big[\chi_x -\gamma_x\sbar(t)\Big]^{x=\tilde{s}(t)}_{x=s(t)}\delta s(t) \,dt,\nonumber
  \\
  R_{17}
  =-\int_{T_1}\psi(s(t),t) \gamma_x(s(t),t)\delta s(t){\delta s}'(t) \,dt,\nonumber
  \\ R_{18}
  =-\int_{T_1} \psi(s(t),t) (au_x)_x\Big\vert^{x=\tilde{s}(t)}_{x=s(t)}\delta s(t)\, dt.\nonumber
\end{gather}
Applying the boundary condition~\eqref{eq:pde-stefan} for $u$ on the moving boundary $s$, as in~\eqref{eq:gradient-heuristic-movingbdy-t2-int-ident}, it follows that
\begin{gather}
  \int_{T_2} \psi a \delta u_x \Big\vert_{x=s(t)}\,dt                =\int_{T_1}\Big[\psi \big(\chi_x {\delta s} - \gamma_x \delta s \sbar'-\gamma {\delta s}'
  - (au_x)_x{\delta s}\big)\Big]_{x=s(t)}\, dt +\nonumber
  \\
  + \sum_{i=19}^{22} R_i,\label{eq:rigorous-movingbdy-t2-int-ident}
  \intertext{where}
  R_{19}
  =\int_{T_2} \psi(s(t),t)\Big[\chi_x- \gamma_x\sbar'(t)\Big]^{x=\tilde{s}(t)}_{x=s(t)}\delta s(t) \,dt,\nonumber
  \\
  R_{20}
  =-\int_{T_2} \psi \gamma_x\big\vert_{x=s(t)} {\delta s}'(t)\delta s(t)\,dt,\nonumber
  \\ R_{21}
  =-\int_{T_2} \psi(s(t),t) \Big[(a \ubar_x)_x\Big]^{x=\tilde{s}(t)}_{x=s(t)} \delta s(t) \,dt,\nonumber
  \\
  R_{22}
  =- \int_{T_2} \psi(s(t),t) \Big[\big(a \ubar_x\big)_x-\big(a u_x\big)_x\Big]_{x=s(t)} \delta s(t) \,dt.\nonumber
\end{gather}
Therefore, taking the sum of $\Delta \J$ using decomposition~\eqref{eq:functional-increment-1} with~\eqref{eq:gradient-heuristic-split-J-1-final},~\eqref{eq:gradient-heuristic-split-J-2-final}, and~\eqref{eq:rigorous-split-J-3-final} and $\Delta I$ using identities~\eqref{eq:rigorous-movingbdy-t1-int-ident} and~\eqref{eq:rigorous-movingbdy-t2-int-ident}, as in formula~\eqref{eq:gradient-heuristic-final-increment} the following formula for the increment of $\J$ is valid\def\numri{22}
\begin{align*}
  \Delta \J(v)
    & =d\J(v) + \sum_{i=1}^{\numri} R_i.
\end{align*}
Denote by $C$ a constant depending on the fixed data $\Omega$, $D$, $V_R$, $a$, etc.
For $R_{22}$, applying Morrey's inequality and Sobolev embedding of Lemma~\ref{lem:w221embedding}, derive
\begin{align*}
  \lnorm{R_{22}}
           & \leq C\norm{\psi}_{\adjointsolnspace(\Omega)} \norm{\delta s}_{W_2^1[0,T]} \int_{T_2} \lnorm{\Big[\big(a \ubar_x\big)_x-\big(a u_x\big)_x\Big]_{x=s(t)}} \,dt.
  \intertext{Expand on the right-hand side as}
           & \leq C\norm{\psi}_{\adjointsolnspace(\Omega)} \norm{\delta s}_{W_2^1[0,T]} \int_{T_2} \lnorm{\Big[a_x \ubar_x-a_x u_x + a \ubar_{xx}-a u_{xx}\Big]_{x=s(t)}} \,dt.
  \intertext{By Minkowski inequality and CBS inequality, it follows that}
  \lnorm{R_{22}}
           & = \norm{\psi}_{\adjointsolnspace(\Omega)} \norm{\delta s}_{W_2^1[0,T]}\sqrt{T}\left[R_{22,1} + R_{22,2}\right],
  \intertext{where}
  R_{22,1}
           & =\left(\int_{T_2} \lnorm{\Big[a_x \ubar_{x}-a_x u_{x}\Big]_{x=s(t)}}^2 \,dt \right)^{1/2},\nonumber
  \\
  R_{22,2}
           & =\left(\int_{T_2} \lnorm{\Big[a \ubar_{xx}-a u_{xx}\Big]_{x=s(t)}}^2 \,dt\right)^{1/2}.\nonumber
  \intertext{Take the transformation $y=x/\sbar(t)$ and apply Minkowski inequality to derive}
  R_{22,2}
  & \leq \sqrt{\norm{a}_{C}}\bigg[\left(\int_{T_2} \lnorm{\Delta \utilde_{xx}(1,t)}^2 \,dt\right)^{1/2} + \nonumber
  \\
           & \quad+ \left(\int_{T_2} \lnorm{\utilde_{xx}(1,t)- \utilde_{xx}(s(t)/\sbar(t),t)}^2 \,dt\right)^{1/2}\bigg].
  \intertext{By trace embedding of Theorem~\ref{thm:solonnikov-traces} with $\l=5/4+\alpha$ and Lemma~\ref{lem:deltau-main-est}, it follows that}
           & \left(\int_{T_2} \lnorm{\Delta \utilde_{xx}(1,t)}^2 \,dt\right)^{1/2} \leq \norm{\Delta \utilde}_{\solnspace(Q_T)} \to 0~\text{as}~\delta v \to 0.
  \intertext{Since $\utilde \in \solnspace(Q_T)$, the trace $\utilde_{xx}$ exists; by uniform convergence of $\sbar$ to $s$ as $\delta v \to 0$, it follows that $s/\sbar \to 1$ as $\delta v \to 0$. Then continuity of $L_2$ norm with respect to shift implies}
           & \left(\int_{T_2} \lnorm{\utilde_{xx}(1,t)- \utilde_{xx}(s(t)/\sbar(t),t)}^2 \,dt\right)^{1/2} \to 0~\text{as}~\delta v \to 0.
\end{align*}
Hence $R_{22,2} \to 0$ as $\delta v \to 0$.
The proof that $R_{22,1} \to 0$ as $\delta v \to 0$ is similar.
It follows that $R_{22} = o(\delta v)$.
Each remainder term $R_i$ is shown to be of higher than linear order with respect to $\delta v$ using a similar application of the energy estimates from
Lemmas~\ref{lem:j-well-defined},~\ref{lem:adjoint-solution-exists},~\ref{lem:deltau-main-est}, trace embedding theorem~\cite[\S 3, thm.\ 9]{solonnikov64}, and the continuity of $L_2$ with respect to shift.
Theorem~\ref{thm:gradient-result2} is proved.

\section{Numerical Method}\label{sec:numerics}
Frechet differentiability result of Theorem~\ref{thm:gradient-result2} and the formula~\eqref{eq:gradient-full} for the Frechet differential suggests the following algorithm based on the projective gradient method:
\begin{itemize}
\item {\bf Step 1.} Set $k=0$ and choose initial vector function $v_0=(f_0,g_0,s_0) \in V_R$.
\item {\bf Step 2.} Solve the PDE problem~\eqref{eq:pde-1}--\eqref{eq:pde-stefan} to find $u_k=u(x,t;v_k)$ and $\J(v_k)$.
\item {\bf Step 3.} If $k=0$, move to Step 4. Otherwise, check the following criteria:
\begin{equation}
\lnorm{\J(v_{k+1})-\J(v_k)}<\epsilon,\quad
\norm{v_{k+1}-v_k} < \epsilon,\label{convergencecriteria}
\end{equation}
where $\epsilon$ is the required accuracy. If criteria is satisfied, then terminate the iteration. Otherwise, move to Step 4.
\item {\bf Step 4.} Having $u_k$, solve the adjoined PDE problem~\eqref{eq:adj-pde}--\eqref{eq:adj-robin-free} to find $\psi_k=\psi(x,t;v_k)$.
\item {\bf Step 5.} Choose stepsize parameter $\alpha_k>0$ and compute new control vector $v_{k+1}=(f_{k+1}, g_{k+1}, s_{k+1})\in V_R$ as follows:
\begin{gather}
  f_{k+1}(x,t)=\mathcal{P}_{V_R}(f_k(x,t)+\alpha_k \psi_k(x,t)),\label{gradientupdate_f}\\
  g_{k+1}(t)=\mathcal{P}_{V_R}(g_k(t)+\alpha_k \psi_k(0,t)),\label{gradientupdate_g}\\
  s_{k+1}(t)=\mathcal{P}_{V_R}\Big(s_k(t)+\alpha_k \Big[2 \beta_1 (u_k-\mu) u_{kx} + \nonumber
\\
   \psi_k \left(\chi_x - \gamma_x s_k' -\big(a u_{kx}\big)_x\right)\Big]_{x=s_k(t)}\Big ),\label{gradientupdate_s}\\
  s'_{k+1}(t)=\mathcal{P}_{V_R}\Big(s'_k(t)+\alpha_k \big[\gamma \psi_k\big]_{x=s_k(t)}\Big),\label{gradientupdate_s'}\\
  s_{k+1}(T)=\mathcal{P}_{V_R}\Big(s_k(T)+\alpha_k [\beta_0\lnorm{u_k(s_k(T),T) - w(s_k(T))}^2 + \nonumber
\\
 2\beta_2(s_k(T)-s_*)]\Big),\label{gradientupdate_s(T)}
\end{gather}
where $\mathcal{P}_{V_R}:H \to V_R$ is the projection operator. Replace $k$ with $k+1$ and move to Step 2.
\end{itemize}
Note that since the free boundary $s$ is a component of the control vector,
direct and adjoined PDE problems are solved in a fixed region at every step of
the iteration. It should be also noted that in the application of this method,
equation \eqref{gradientupdate_s} must be coordinated with the equations
\eqref{gradientupdate_s'} and \eqref{gradientupdate_s(T)}. We refer the
implementation of this algorithm and analysis of the numerical results for model
examples in a subsequent paper.

\section{Conclusions}\label{conclusions}
Following the new variational formulation of the inverse Stefan problem introduced in \cite{abdulla13,abdulla15}, optimal control of the second order parabolic free boundary problem is analyzed in the Besov spaces framework in this paper. Existence of optimal control and Frechet differentiability is proved, and the formula for the Frechet differential is derived under minimal regularity assumptions on the data. The result implies a necessary condition for the optimality in the form of variational inequality, and opens a way for the implementation of an effective numerical method based on the projective gradient method in Besov spaces framework. The main idea of the new variational formulation is optimal control setting, where the free boundary is the component of the control vector. This allows for the development of an iterative gradient type numerical method of low computational cost. It also creates a frame for the regularization of the error existing in the information on the phase transition temperature.

 \section*{Acknowledgement}
This research was funded by National Science Foundation: grant \#1359074--REU Site: Partial Differential Equations and Dynamical Systems at Florida Instiute of Technology.
Two REU students Jessica Pillow and Dylanger Pittman worked on part of the project restricted to heuristic derivation given in Section~\ref{sec:gradient-heuristic} with the intention to implement gradient type method for numerical analysis.


\begin{thebibliography}{10}
\providecommand{\url}[1]{{#1}}
\providecommand{\urlprefix}{URL }
\expandafter\ifx\csname urlstyle\endcsname\relax
  \providecommand{\doi}[1]{DOI~\discretionary{}{}{}#1}\else
  \providecommand{\doi}{DOI~\discretionary{}{}{}\begingroup
  \urlstyle{rm}\Url}\fi

\bibitem{abdulla13}
Abdulla, U.: On the optimal control of the free boundary problems for the
  second order parabolic equations. {I}. well-posedness and convergence of the
  method of lines.
\newblock Inverse Problems and Imaging \textbf{7}(2), 307--340 (2013).
\newblock \doi{10.3934/ipi.2013.7.307}

\bibitem{abdulla15}
Abdulla, U.: On the optimal control of the free boundary problems for the
  second order parabolic equations. {II}. {C}onvergence of the method of finite
  differences.
\newblock Inverse Problems and Imaging (To Appear)  (2016).
\newblock \urlprefix\url{http://arxiv.org/abs/1506.02341}

\bibitem{cannon67}
Cannon, J.R., Jr., J.D.: The {Cauchy} problem for the heat equation.
\newblock SIAM Journal on Numerical Analysis \textbf{4}(3), 317--336 (1967)

\bibitem{budak72}
Budak, B.M., Vasileva, V.N.: On the solution of the inverse {Stefan} problem.
\newblock Soviet Mathematics Doklady \textbf{13}, 811--815 (1972)

\bibitem{budak73}
Budak, B.M., Vasileva, V.N.: On the solution of {Stefan}'s converse problem
  {II}.
\newblock USSR Computational Mathematics and Mathematical Physics \textbf{13},
  97--110 (1973)

\bibitem{budak74}
Budak, B.M., Vasileva, V.N.: The solution of the inverse {Stefan} problem.
\newblock USSR Computational Mathematics and Mathematical Physics
  \textbf{13}(1), 130--151 (1974).
\newblock \doi{10.1016/0041-5553(74)90010-X}

\bibitem{vasilev69}
Vasil'ev, F.P.: The existence of a solution to a certain optimal {Stefan}
  problem.
\newblock Computational Methods and Programming pp. 110--114 (1969)

\bibitem{yurii80}
Yurii, A.D.: On an optimal {Stefan} problem.
\newblock Doklady Akademii nauk SSSR \textbf{251}, 1317--1321 (1980)

\bibitem{alifanov95}
Alifanov, O., Artyukhin, E., Rumiantsev, S.: Extreme Methods for Solving
  Ill-Posed Problems with Applications to Inverse Problems.
\newblock Begell Publ. House (1995)

\bibitem{bell81}
Bell, J.B.: The non-characteristic {Cauchy} problem for a class of equations
  with time dependence. {I.} problem in one space dimension.
\newblock SIAM Journal on Mathematical Analysis \textbf{12}(5), 759--777 (1981)

\bibitem{cannon64}
Cannon, J.R.: A {Cauchy} problem for the heat equation.
\newblock Annali di Matematica Pura Ed Applicata \textbf{66}(1), 155--165
  (1964).
\newblock \doi{10.1007/BF02412441}

\bibitem{carasso82}
Carasso, A.: Determining surface temperatures from interior observations.
\newblock SIAM Journal on Applied Mathematics \textbf{42}(3), 558--574 (1982)

\bibitem{ewing79}
Ewing, R.E.: The {Cauchy} problem for a linear parabolic equation.
\newblock Journal of Mathematical Analysis and Applications \textbf{71}(1),
  167--186 (1979).
\newblock \doi{10.1016/0022-247X(79)90223-3}

\bibitem{ewing79a}
Ewing, R.E., Falk, R.S.: Numerical approximation of a {Cauchy} problem for a
  parabolic partial differential equations.
\newblock Mathematics of Computation \textbf{33}(148), 1125--1144 (1979)

\bibitem{goldman97}
Gol'dman, N.: Inverse {Stefan} Problems.
\newblock Kluwer Academic Publishers Group, Dodrecht (1997)

\bibitem{hoffman81}
Hoffman, K.H., Niezgodka, M.: Control of parabolic systems involving free
  boundaries.
\newblock In: Proceedings of the International Conference on Free Boundary
  Problems (1981)

\bibitem{sherman71}
Sherman, B.: General one-phase {Stefan} problems and free boundary problems for
  the heat equation with {Cauchy} data prescribed on the free boundary.
\newblock SIAM J. Appl. Math. \textbf{20}, 557--570 (1971)

\bibitem{baumeister80}
Baumeister, J.: Zur optimal {Steuerung} von frien {Randwertausgaben}.
\newblock ZAMM \textbf{60}, 335--339 (1980)

\bibitem{fasano77}
Fasano, A., Primicerio, M.: General free boundary problems for heat equations.
\newblock Journal of Mathematical Analysis and Applications \textbf{57}(3),
  694--723 (1977)

\bibitem{hoffman82}
Hoffman, K.H., Sprekels, J.: Real time control of free boundary in a two-phase
  {Stefan} problem.
\newblock Numerical Functional Analysis and Optimization \textbf{5}, 47--76
  (1982)

\bibitem{hoffman86}
Hoffman, K.H., Sprekels, J.: On the identification of heat conductivity and
  latent heat conductivity as latent heat in a one-phase {Stefan} problem.
\newblock Control and Cybernetics \textbf{15}, 37--51 (1986)

\bibitem{jochum80}
Jochum, P.: The numerical solution of the inverse {Stefan} problem.
\newblock Numerical Mathematics \textbf{34}, 411--429 (1980)

\bibitem{jochum80a}
Jochum, P.: The inverse {Stefan} problem as a problem of nonlinear
  approximation theory.
\newblock Journal of Approximation Theory \textbf{30}, 37--51 (1980)

\bibitem{knabner83}
Knabner, P.: Stability theorems for general free boundary problems of the
  {Stefan} type and applications.
\newblock Applied Nonlinear Functional Analysis, Methoden und Verfahren der
  Mathematischen Physik \textbf{25}, 95--116 (1983)

\bibitem{ladyzhenskaya68}
Ladyzhenskaya, O.A., Solonnikov, V.A., Uraltseva, N.N.: Linear and Quasilinear
  Equations of Parabolic Type, \emph{Translations of Mathematical Monographs},
  vol.~23.
\newblock American Mathematical Society, Providence, R. I. (1968)

\bibitem{lurye75}
Lurye, K.A.: Optimal Control in Problems of Mathematical Physics.
\newblock Moscow. Nauka (1975)

\bibitem{niezgodka79}
Niezgodka, M.: Control of parabolic systems with free boundaries - application
  of inverse formulation.
\newblock Control and Cybernetics \textbf{8}, 213--225 (1979)

\bibitem{nochetto87}
Nochetto, R.H., Verdi, C.: The combined use of nonlinear {Chernoff} formula
  with a regularization procedure for two-phase {Stefan} problems.
\newblock Numerical Functional Analysis and Optimization \textbf{9}, 1177--1192
  (1987/88)

\bibitem{primicerio82}
Primicerio, M.: The occurence of pathologies in some {Stefan}-like problems.
\newblock In: J.~Albrecht, L.~Collatz, K.H. Hoffman (eds.) Numerical Treatment
  of Free Boundary-Value problems, vol.~58, pp. 233--244. ISNM, Birkhauser
  Verlag, Basel (1982)

\bibitem{sagues82}
Sagues, C.: Simulation and optimal control of free boundary.
\newblock In: J.~Albrecht, L.~Collatz, K.H. Hoffman (eds.) Numerical Treatment
  of Free Boundary-Value problems, vol.~58, pp. 270--287. ISNM, Birkhauser
  Verlag, Basel (1982)

\bibitem{talenti82}
Talenti, G., Vessella, S.: A note on an ill-posed problem for the heat
  equation.
\newblock Journal of the Austrailian Mathematical Society \textbf{32}(3),
  358--368 (1982).
\newblock \doi{10.1017/S1446788700024915}

\bibitem{besov79}
Besov, O.V., Ilin, V.P., Nikolskii, S.M.: Integral Representations of Functions
  and Imbedding Theorems, vol.~1.
\newblock John Wiley \& Sons (1979)

\bibitem{besov79a}
Besov, O.V., Ilin, V.P., Nikolskii, S.M.: Integral Representations of Functions
  and Imbedding Theorems, vol.~2.
\newblock John Wiley \& Sons (1979)

\bibitem{kufner77}
Kufner, A., John, O., Fu\v{c}ik, S.: Function Spaces.
\newblock Noordhoff International Publishing, Leyden, The Netherlands (1977)

\bibitem{nikolskii75}
Nikol'skii, S.M.: Approximation of Functions of Several Variables and Imbedding
  Theorems.
\newblock Springer-Verlag, New York-Heidelberg (1975)

\bibitem{solonnikov64}
Solonnikov, V.A.: A-priori estimates for solutions of second-order equations of
  parabolic type, \emph{Trudy Matematischeskogo instituta im. V. A. Steklova},
  vol.~70.
\newblock Nauka, Moscow-Leningrad (1964)

\bibitem{solonnikov65}
Solonnikov, V.A.: On boundary value problems for linear parabolic systems of
  differential equations in general form.
\newblock Proceedings of the Steklov Institute of Mathematics \textbf{83},
  1--184 (1965)

\bibitem{yosida96}
Yosida, K.: Functional Analysis.
\newblock Classics in Mathematics. Springer (1996)

\end{thebibliography}
\end{document}